\documentclass{article}
\usepackage[utf8]{inputenc}
\usepackage{amsmath}
\usepackage{amsthm}
\usepackage{amssymb}
\usepackage{array}
\usepackage{enumitem}
\usepackage{tikz}
\usepackage{array}
\usepackage{biblatex}
\usetikzlibrary{shapes}
\usetikzlibrary{arrows}
\usetikzlibrary{matrix}
\theoremstyle{definition}\newtheorem{prop}{Proposition}[section]
\theoremstyle{definition}
\theoremstyle{definition}\newtheorem{definition}[prop]{Definition}
\theoremstyle{definition}
\theoremstyle{definition}\newtheorem{corollary}[prop]{Corollary}
\theoremstyle{definition}\newtheorem{theorem}[prop]{Theorem}
\theoremstyle{definition}\newtheorem{example}[prop]{Example}

\title{Independence equivalence classes of paths}
\author{Boon Leong Ng}
\date{%
    National Institute of Education, Nanyang Technological University, Singapore\\
}

\begin{document}

\maketitle

\begin{abstract}
The independence equivalence class of a graph $G$ is the set of graphs that have the same independence polynomial as $G$. A graph whose independence equivalence class contains only itself, up to isomorphism, is independence unique. Beaton, Brown and Cameron~(2019) showed that paths with an odd number of vertices are independence unique and raised the problem of finding the independence equivalence class of paths with an even number of vertices. The problem is completely solved in this paper.
\end{abstract}

\section{Definitions and Introduction}
Let $G$ be a simple graph with vertex set $V(G)$ and edge set $E(G)$. A set of vertices $U\subseteq V(G)$ is said to be \emph{independent} if no two vertices in $U$ are adjacent. The \emph{independence number} $\alpha(G)$ of $G$ is the cardinality of the largest independent set of $G$. The \emph{independence polynomial} $I(G,x)$ of $G$ is given by $I(G,x) = \sum_{k=0}^{\alpha(G)} i_k(G) x^k$ where $i_k(G)$ is the number of independent subsets of $V(G)$ of cardinality $k$. A survey of results on independence polynomials can be found in~\cite{levit}.

Two graphs $G$ and $H$ are said to be \emph{independence equivalent} if $I(G,x) = I(H,x)$. The \emph{independence equivalence class} $\mathcal{I}(G)$ of a graph $G$ is the set of graphs which are independence equivalent to $G$. A graph $G$ is said to be \emph{independence unique} if $\mathcal{I}(G)=\{G\}$, that is, if $I(G,x) = I(H,x)$ implies that $G\cong H$ ($G$ is isomorphic to $H$).

The existence of independence equivalent graphs was already noticed very early on by Wingard~\cite{wingard}. Wagner and Wang~\cite{huawang} showed that almost all trees have a (nonisomorphic) independence equivalent tree, in the sense that as $n\rightarrow\infty$, the proportion of trees of order~$n$ which have an independence equivalent tree, out of all trees of order~$n$, tend to~$1$. Makowsky and Rakita~\cite{makowsky} would subsequently show that almost all graphs are not independence unique.

The following two results on finding the independence polynomial of a graph in terms of that of its subgraphs can be found in~\cite{hoedeli}.
\begin{prop}\label{prop:delvert}
For any vertex $u\in V(G)$,
$$I(G, x) = I(G-u, x)+xI(G-N[u], x).$$
\end{prop}
\begin{prop}\label{prop:deledge}
For any edge $e=uv$ in $E(G)$,
$$I(G, x) = I(G-e, x)-x^2 I(G-(N(u)\cup N(v)), x).$$
\end{prop}
\begin{figure}
\centering
\begin{tikzpicture}
\filldraw(0,0) circle[radius=2pt]node[left]{$a$};
\filldraw(0,2) circle[radius=2pt];
\filldraw(1,1) circle[radius=2pt]node[above]{$b$};
\filldraw(2,1) circle[radius=2pt];
\node at (3,1) {$\cdots$};
\filldraw(4,1) circle[radius=2pt];
\draw(0,0)--(0,2);
\draw(0,0)--(1,1);
\draw(0,2)--(1,1);
\draw(1,1)--(2,1);
\draw(2,1)--(2.5,1);
\draw(3.5,1)--(4,1);
\node at (0.6,0.4) {$f$};
\end{tikzpicture}
\caption{The graph $D_n$ with $n$ vertices}\label{fig:dn}
\end{figure}
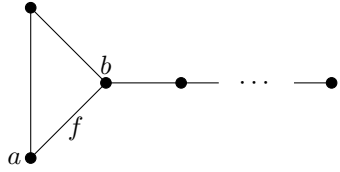
These are very useful in showing that graphs are independence equivalent. For example, the following family of independence equivalent graphs was observed by Chism~\cite{chism}.

\begin{prop}\cite{chism}\label{prop:cndn}
For $n\geq 4$, if $G$ is the cycle graph $C_n$ and $H$ is the graph $D_n$ shown in Figure~\ref{fig:dn}, then $G$ and $H$ are independence equivalent.
\end{prop}
%\begin{proof}
%For any vertex $u\in V(G)$, $G-u \cong H-a \cong P_{n-1}$ and $G-N[u] \cong H-N[a] \cong P_{n-3}$, so by Proposition~\ref{prop:delvert}, $I(G,x)=I(H,x)$. 

%Alternatively, for any edge $e=uv\in E(G)$, $G-e \cong H-f \cong P_n$ and $G-(N(u)\cup N(v)) \cong H-(N(a)\cup N(b)) \cong P_{n-4}$, so by Proposition~\ref{prop:deledge}, $I(G,x)=I(H,x)$.
%\end{proof}

From the definition of the independence polynomial, the following proposition is clear.
\begin{prop}\label{prop:multi}
If $G$ is a disconnected graph with connected components $G_1,\ldots,G_r$, then
$$I(G,x)=\prod_{i=1}^r I(G_i,x).$$
\end{prop}
In other words, the independence polynomial of a graph $G$ is the product of the independence polynomials of the connected components of $G$. The factorisation (or lack thereof) of $I(G,x)$ over $\mathbb{Z}[x]$ is therefore a useful tool in determining $\mathcal{I}(G)$. We will explore this in detail in Section~\ref{sec:factor}.

Wingard showed that any tree which is independence equivalent to the path $P_n$ must be isomorphic to $P_n$ (Corollary~5.6 of~\cite{wingard}). Since we can determine the number of vertices and edges from the coefficients of the independence polynomial, it follows that any connected graph which is independence equivalent to $P_n$ must also be a tree, and hence isomorphic to $P_n$. Therefore, we need only consider non-connected graphs. One such graph was discovered by Zhang~\cite{zhang}.

\begin{prop}\label{prop:zhang}
For $n\geq 2$, if $G$ is the path $P_{2n}$ and $H$ is the graph $P_{n-1}\cup C_{n+1}$, then $G$ and $H$ are independence equivalent.
\end{prop}

%\begin{proof}
%For vertex $u\in V(G)$ be one of the two central vertices, and vertex $v\in V(H)$ be any vertex on the $C_{n+1}$ component. Then $G-u \cong P_{n-1}\cup P_{n} \cong H-v$ and $G-N[u] \cong P_{n-2}\cup P{n-1} \cong H-N[v]$.
%\end{proof}

It is because of this proposition that the problem of finding the independence equivalence class of (even) paths becomes intimately tied to that of finding the independence equivalence class of cycle graphs. The problem of finding the independence equivalence class of even cycle graphs was solved by Beaton, Brown and Cameron~\cite{beaton}, and that of odd cycle graphs by Ng~\cite{ng}.

\begin{theorem}\cite{beaton,ng}\label{prop:cn}
\begin{enumerate}[label=(\roman*)]
    \item $\mathcal{I}(C_3)=\{C_3\}$,
    \item $\mathcal{I}(C_6)=\{C_6, D_6, K_4-e \cup P_2\}$,
    \item $\mathcal{I}(C_9)=\{C_9, D_9, C_3 \cup G_a, C_3 \cup G_b, C_3 \cup G_c, C_3 \cup G_d\}$  (see Figure~\ref{fig:c9}),
    \item $\mathcal{I}(C_{15})=\{C_{15}, D_{15}, C_3 \cup C_5 \cup G^\prime_a, C_3 \cup D_5 \cup G^\prime_a, C_3 \cup C_5 \cup G^\prime_b, C_3 \cup D_5 \cup G^\prime_b, C_3 \cup C_5 \cup G^\prime_c, C_3 \cup D_5 \cup G^\prime_c\}$ (see Figure~\ref{fig:c15}),
    \item $\mathcal{I}(C_{n})=\{C_n,D_n\}$ for $n\geq 4$ and $n\neq 6,9,15$.
\end{enumerate}
\end{theorem}

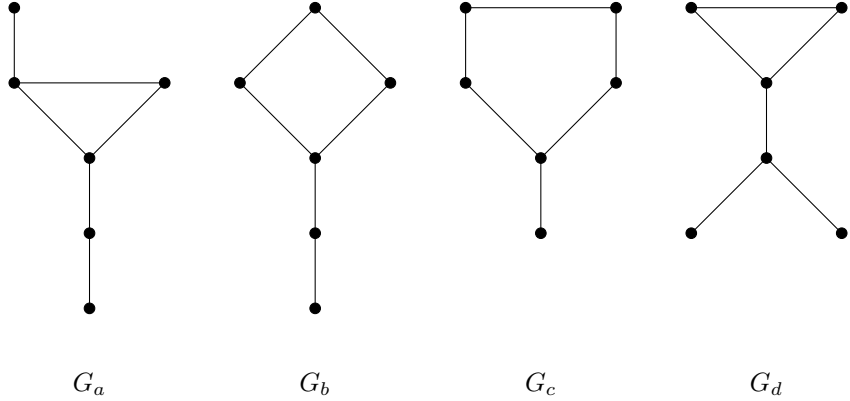
\begin{figure}
\centering
\begin{tikzpicture}
\filldraw(0,4) circle[radius=2pt];
\filldraw(0,3) circle[radius=2pt];
\filldraw(2,3) circle[radius=2pt];
\filldraw(1,2) circle[radius=2pt];
\filldraw(1,1) circle[radius=2pt];
\filldraw(1,0) circle[radius=2pt];
\draw(0,4)--(0,3);
\draw(0,3)--(2,3);
\draw(0,3)--(1,2);
\draw(2,3)--(1,2);
\draw(1,2)--(1,1);
\draw(1,1)--(1,0);
\node at (1,-1) {$G_a$};

\filldraw(4,4) circle[radius=2pt];
\filldraw(3,3) circle[radius=2pt];
\filldraw(5,3) circle[radius=2pt];
\filldraw(4,2) circle[radius=2pt];
\filldraw(4,1) circle[radius=2pt];
\filldraw(4,0) circle[radius=2pt];
\draw(4,4)--(3,3);
\draw(4,4)--(5,3);
\draw(3,3)--(4,2);
\draw(5,3)--(4,2);
\draw(4,2)--(4,1);
\draw(4,1)--(4,0);
\node at (4,-1) {$G_b$};

\filldraw(6,4) circle[radius=2pt];
\filldraw(8,4) circle[radius=2pt];
\filldraw(6,3) circle[radius=2pt];
\filldraw(8,3) circle[radius=2pt];
\filldraw(7,2) circle[radius=2pt];
\filldraw(7,1) circle[radius=2pt];
\draw(6,4)--(8,4);
\draw(6,4)--(6,3);
\draw(8,4)--(8,3);
\draw(6,3)--(7,2);
\draw(8,3)--(7,2);
\draw(7,2)--(7,1);
\node at (7,-1) {$G_c$};

\filldraw(9,4) circle[radius=2pt];
\filldraw(11,4) circle[radius=2pt];
\filldraw(10,3) circle[radius=2pt];
\filldraw(10,2) circle[radius=2pt];
\filldraw(9,1) circle[radius=2pt];
\filldraw(11,1) circle[radius=2pt];
\draw(9,4)--(11,4);
\draw(9,4)--(10,3);
\draw(11,4)--(10,3);
\draw(10,3)--(10,2);
\draw(10,2)--(9,1);
\draw(10,2)--(11,1);
\node at (10,-1) {$G_d$};
\end{tikzpicture}
\caption{Graphs $G_a$, $G_b$, $G_c$ and $G_d$ in Theorem~\ref{prop:cn}}\label{fig:c9}
\end{figure}

\begin{figure}
\centering
\begin{tikzpicture}
\filldraw(0,4) circle[radius=2pt];
\filldraw(0,3) circle[radius=2pt];
\filldraw(2,3) circle[radius=2pt];
\filldraw(1,2) circle[radius=2pt];
\filldraw(1,1) circle[radius=2pt];
\filldraw(1,0) circle[radius=2pt];
\filldraw(1,-1) circle[radius=2pt];
\draw(0,4)--(0,3);
\draw(0,3)--(2,3);
\draw(0,3)--(1,2);
\draw(2,3)--(1,2);
\draw(1,2)--(1,1);
\draw(1,1)--(1,0);
\draw(1,0)--(1,-1);
\node at (1,-2) {$G^\prime_a$};

\filldraw(5,4) circle[radius=2pt];
\filldraw(4,3) circle[radius=2pt];
\filldraw(6,3) circle[radius=2pt];
\filldraw(5,2) circle[radius=2pt];
\filldraw(5,1) circle[radius=2pt];
\filldraw(5,0) circle[radius=2pt];
\filldraw(5,-1) circle[radius=2pt];
\draw(5,4)--(4,3);
\draw(5,4)--(6,3);
\draw(4,3)--(5,2);
\draw(6,3)--(5,2);
\draw(5,2)--(5,1);
\draw(5,1)--(5,0);
\draw(5,0)--(5,-1);
\node at (5,-2) {$G^\prime_b$};

\filldraw(9,4) circle[radius=2pt];
\filldraw(8,3) circle[radius=2pt];
\filldraw(10,3) circle[radius=2pt];
\filldraw(8,2) circle[radius=2pt];
\filldraw(10,2) circle[radius=2pt];
\filldraw(9,1) circle[radius=2pt];
\filldraw(9,0) circle[radius=2pt];
\draw(9,4)--(8,3);
\draw(9,4)--(10,3);
\draw(8,3)--(8,2);
\draw(10,3)--(10,2);
\draw(8,2)--(9,1);
\draw(10,2)--(9,1);
\draw(9,1)--(9,0);
\node at (9,-2) {$G^\prime_c$};
\end{tikzpicture}
\caption{Graphs $G^\prime_a$, $G^\prime_b$ and $G^\prime_c$ in Theorem~\ref{prop:cn}}\label{fig:c15}
\end{figure}
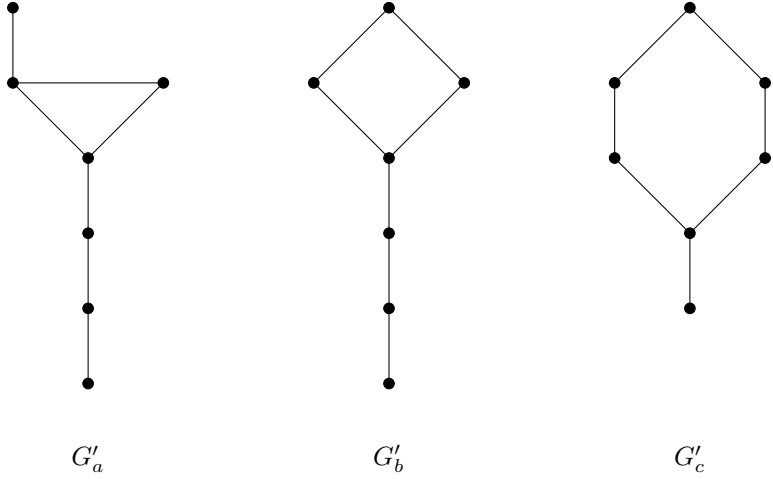

In the case of paths, Beaton, Brown and Cameron~\cite{beaton} showed that $P_n$ is independence unique for odd~$n$.

For even $n$, we can make use of Propositions~\ref{prop:zhang} and~\ref{prop:cn} to list out some graphs which are independence equivalent to $P_n$. If $n+2=2^t m$, where $t\geq 1$ and $m$ is odd, then, by Proposition~\ref{prop:zhang}, the following graphs are in $\mathcal{I}(P_n)$ (note that in the special case $m=1$, we only have $Z_{n,1}$ to $Z_{n,t-2}$):
\begin{eqnarray*}
Z_{n,0} & = & P_{2^{t} m-2},\\
Z_{n,1} & = & C_{2^{t-1} m} \cup P_{2^{t-1} m-2},\\
Z_{n,2} & = & C_{2^{t-1} m} \cup C_{2^{t-2} m} \cup P_{2^{t-2} m-2}, \\
&\vdots&\\
Z_{n,t-2} & = & C_{2^{t-1} m} \cup C_{2^{t-2} m} \cup \cdots \cup C_{4m} \cup P_{4m-2}, \\
Z_{n,t-1} & = & C_{2^{t-1} m} \cup C_{2^{t-2} m} \cup \cdots \cup C_{4m} \cup C_{2m} \cup P_{2m-2}, \\
Z_{n,t} & = & C_{2^{t-1} m} \cup C_{2^{t-2} m} \cup \cdots \cup C_{4m} \cup C_{2m} \cup C_{m} \cup P_{m-2}.
\end{eqnarray*}
This motivates the following definition of a set of independence equivalent graphs.
\begin{definition}\label{def:mathcald}
Let $G=H \cup C_{k_1} \cup C_{k_2} \cup \cdots \cup C_{k_m}$ where $H$ does not have any connected components which are cycles. Then we define
$$\mathcal{D}(G) = \left\{ H \cup U_{k_1} \cup U_{k_2} \cup \cdots \cup U_{k_m} ~\vert~ U_{k_i}\in\mathcal{I}(C_{k_i})~\forall~i\in\{1,\ldots,m\} \right\}.$$
%For any graph $G$, let $\mathcal{D}(G)$ be the set of graphs obtained by replacing some (possibly all or none) of the connected components of $G$ which are cycle graphs $C_k$, with another member of $\mathcal{I}(C_k)$. 
\end{definition}

It follows from Proposition~\ref{prop:multi} that $\mathcal{D}(G)\subseteq\mathcal{I}(G)$. The following examples make use of Proposition~\ref{prop:cn}.

\begin{example}
$\mathcal{D}(Z_{18,2})=\{C_{10} \cup C_5 \cup P_3, D_{10} \cup C_5 \cup P_3, C_{10} \cup D_5 \cup P_3, D_{10} \cup D_5 \cup P_3\}$.
\end{example}

\begin{example}
$\mathcal{D}(Z_{10,2})=\{C_6 \cup C_3 \cup P_1, D_6 \cup C_3 \cup P_1, (K_4-e) \cup P_2 \cup C_3 \cup P_1\}$.
\end{example}

\begin{example}
$\mathcal{D}(Z_{16,1})=\{C_9 \cup P_7, D_9 \cup P_7, C_3 \cup G_a \cup P_7, C_3 \cup G_b \cup P_7, C_3 \cup G_c \cup P_7, C_3 \cup G_d \cup P_7\}$  (see Figure~\ref{fig:c9}).
\end{example}

\begin{example}
$\mathcal{D}(Z_{28,1})=\{C_{15} \cup P_{13}, D_{15} \cup P_{13}, C_3 \cup C_5 \cup G^\prime_a \cup P_{13},  C_3 \cup D_5 \cup G^\prime_a \cup P_{13}, C_3 \cup C_5 \cup G^\prime_b \cup P_{13},  C_3 \cup D_5 \cup G^\prime_b \cup P_{13}, C_3 \cup C_5 \cup G^\prime_c \cup P_{13},  C_3 \cup D_5 \cup G^\prime_c \cup P_{13}\}$  (see Figure~\ref{fig:c15}).
\end{example}

\begin{definition}\label{def:mathcalp}
Let $n$ be a positive even integer and $n+2=2^t m$, where $t$ is a positive integer and $m$ is odd. Define
$$Z_{n,i}=C_{2^{t-1} m} \cup C_{2^{t-2} m} \cup \ldots \cup C_{2^{t-i} m} \cup P_{2^{t-i} m-2},$$
and
$$\mathcal{P}_n=
     \begin{cases}
       \{P_n\}\cup\bigcup_{i=1}^t \mathcal{D}(Z_{n,i}) &\quad\text{if }m\geq3, \\
       \{P_n\}\cup\bigcup_{i=1}^{t-2} \mathcal{D}(Z_{n,i}) &\quad\text{if }m=1. \\
     \end{cases}.$$
\end{definition}
Then $\mathcal{P}_n\subseteq\mathcal{I}(P_n)$.

We will eventually show (Theorem~\ref{prop:pathequiv}) that $\mathcal{P}_n = \mathcal{I}(P_n)$ for all odd $m$ except~$3$. The complexity of the $m=3$ case was already hinted at by Beaton, Brown and Cameron~\cite{beaton} when they showed that $\mathcal{I}(P_{10})$ contains ten graphs.

The set $\mathcal{P}_{10}$ contains seven graphs. The remaining three graphs arise from the fact that $P_1\cup C_6$ is independence equivalent to $Y_{3,2,1}$, $P_2\cup E_{1,1}$, and $P_2 \cup A_{1,1}$ (see Figure~\ref{fig:p10}).

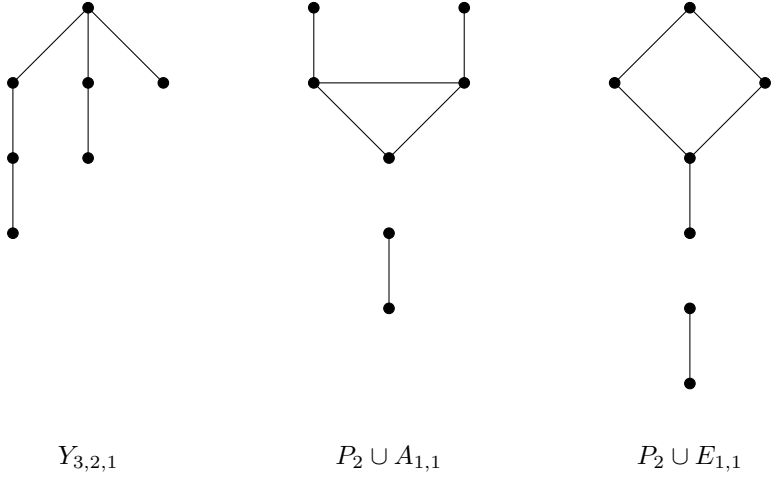
\begin{figure}
\centering
\begin{tikzpicture}
\filldraw(1,6) circle[radius=2pt];
\filldraw(0,5) circle[radius=2pt];
\filldraw(1,5) circle[radius=2pt];
\filldraw(2,5) circle[radius=2pt];
\filldraw(1,4) circle[radius=2pt];
\filldraw(0,4) circle[radius=2pt];
\filldraw(0,3) circle[radius=2pt];
\draw(0,5)--(1,6);
\draw(1,5)--(1,6);
\draw(2,5)--(1,6);
\draw(1,4)--(1,5);
\draw(0,4)--(0,5);
\draw(0,3)--(0,4);
\node at (1,0) {$Y_{3,2,1}$};

\filldraw(4,6) circle[radius=2pt];
\filldraw(6,6) circle[radius=2pt];
\filldraw(4,5) circle[radius=2pt];
\filldraw(6,5) circle[radius=2pt];
\filldraw(5,4) circle[radius=2pt];
\filldraw(5,3) circle[radius=2pt];
\filldraw(5,2) circle[radius=2pt];
\draw(4,6)--(4,5);
\draw(6,6)--(6,5);
\draw(4,5)--(6,5);
\draw(4,5)--(5,4);
\draw(6,5)--(5,4);
\draw(5,3)--(5,2);
\node at (5,0) {$P_2\cup A_{1,1}$};

\filldraw(9,6) circle[radius=2pt];
\filldraw(8,5) circle[radius=2pt];
\filldraw(10,5) circle[radius=2pt];
\filldraw(9,4) circle[radius=2pt];
\filldraw(9,3) circle[radius=2pt];
\filldraw(9,2) circle[radius=2pt];
\filldraw(9,1) circle[radius=2pt];
\draw(9,6)--(8,5);
\draw(9,6)--(10,5);
\draw(8,5)--(9,4);
\draw(10,5)--(9,4);
\draw(9,4)--(9,3);
\draw(9,2)--(9,1);
\node at (9,0) {$P_2 \cup E_{1,1}$};
\end{tikzpicture}
\caption{Graphs $Y_{3,2,1}$, $P_2\cup A_{1,1}$ and $P_2 \cup E_{1,1}$}\label{fig:p10}
\end{figure}

The independence equivalence of $P_1\cup C_6$ and $Y_{3,2,1}$ is a special case of a more general statement, that $P_1\cup C_k$ and $Y_{k-3,2,1}$ (see Figure~\ref{fig:y321}) are independence equivalent for all $k\geq 4$. This can be observed by applying Proposition~\ref{prop:delvert} to one of the vertices in $C_k$ and the vertex adjacent to vertices of degree~3 and~1 in $Y_{k-3,2,1}$. Now, for $n=3\times2^t - 2$, where $t\geq 2$, the graph $Z_{n,t}$ is
$$C_{3\times 2^{t-1}} \cup C_{3\times 2^{t-2}} \cup \cdots \cup C_{6} \cup C_{3} \cup P_{1}.$$

\begin{figure}
\centering
\begin{tikzpicture}
\filldraw(0,1) circle[radius=2pt];
\filldraw(1,0) circle[radius=2pt];
\filldraw(1,1) circle[radius=2pt];
\filldraw(1,2) circle[radius=2pt]node[above]{$v_1$};
\filldraw(2,1) circle[radius=2pt];
\filldraw(2,2) circle[radius=2pt]node[above]{$v_2$};
\filldraw(4,2) circle[radius=2pt]node[above]{$v_{k-3}$};
\draw(0,1)--(1,0);
\draw(0,1)--(1,1);
\draw(0,1)--(1,2);
\draw(1,1)--(2,1);
\draw(1,2)--(2,2);
\draw(2,2)--(2.5,2);
\draw(3.5,2)--(4,2);
\node at (3,2) {$\cdots$};
\end{tikzpicture}
\caption{$Y_{k-3,2,1}$}\label{fig:y321}
\end{figure}
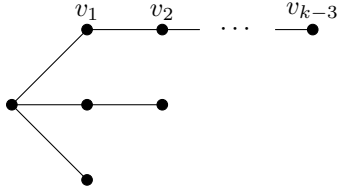

We can thus replace the disjoint union of $P_{1}$ and any of the cycle graphs $C_{3\times 2^z}$ in $Z_{n,t}$ with $Y_{3\times 2^z-3,2,1}$ to obtain an independence equivalent graph.

\begin{definition}\label{def:mathcaly1}
Let $n$ be a positive integer such that $n=3\times2^t - 2$ for some positive integer $t$. Define
$$\mathcal{Y}^{(1)}_n = \begin{cases}
    \displaystyle\bigcup_{z=1}^{t-1} \mathcal{D}\left(Y_{3\times 2^z-3,2,1} \cup \displaystyle\bigcup_{\substack{0\leq i\leq t-1,\\i\neq z}} C_{3\times 2^i}\right) &\quad\text{if }t\geq2, \\
    \varnothing &\quad\text{if }t=1.
\end{cases}$$
\end{definition}
The independence equivalence of $P_2\cup E_{1,1}$ and $P_2 \cup A_{1,1}$ with $P_1\cup C_6$ can be thought of as a sporadic occurrence. As $P_1\cup C_6$ will appear in $Z_{n,t}$ for $n=3\times2^t - 2$ when $t\geq 2$, we have to take them into consideration.
\begin{definition}\label{def:mathcaly23}
Let $n$ be a positive integer such that $n=3\times2^t - 2$ for some positive integer $t$. Define
$$\mathcal{Y}^{(2)}_n = \begin{cases}
    \mathcal{D}\left(P_2\cup E_{1,1} \cup C_3 \cup \displaystyle\bigcup_{i=2}^{t-1} C_{3\times 2^i}\right) &\quad\text{if }t\geq3, \\
    \mathcal{D}\left(P_2\cup E_{1,1} \cup C_3\right) &\quad\text{if }t=2, \\
    \varnothing &\quad\text{if }t=1,
\end{cases}$$
and
$$\mathcal{Y}^{(3)}_n = \begin{cases}
    \mathcal{D}\left(P_2\cup A_{1,1} \cup C_3 \cup \displaystyle\bigcup_{i=2}^{t-1} C_{3\times 2^i}\right) &\quad\text{if }t\geq3, \\
    \mathcal{D}\left(P_2\cup A_{1,1} \cup C_3\right) &\quad\text{if }t=2, \\
    \varnothing &\quad\text{if }t=1.
\end{cases}$$
\end{definition}
Another such sporadic occurrence takes place for $n=3\times2^t - 2$ when $t\geq 3$. In this case, $P_1\cup C_{12}$ appears in $Z_{n,t}$ and it is independence equivalent to $C_4\cup Y_{4,2,2}$ (see Figure~\ref{fig:y422c4}). Since $C_6$ appears in $Z_{n,t}$, and it can be replaced by $(K_4-e) \cup P_2$, we also need to note that $C_4 \cup P_2$ is independence equivalent to $P_6$.

\begin{figure}
\centering
\begin{tikzpicture}
\filldraw(0,1) circle[radius=2pt];
\filldraw(1,0) circle[radius=2pt];
\filldraw(1,1) circle[radius=2pt];
\filldraw(1,2) circle[radius=2pt];
\filldraw(2,0) circle[radius=2pt];
\filldraw(2,1) circle[radius=2pt];
\filldraw(2,2) circle[radius=2pt];
\filldraw(3,2) circle[radius=2pt];
\filldraw(4,2) circle[radius=2pt];
\draw(0,1)--(1,0);
\draw(0,1)--(1,1);
\draw(0,1)--(1,2);
\draw(1,0)--(2,0);
\draw(1,1)--(2,1);
\draw(1,2)--(2,2);
\draw(2,2)--(3,2);
\draw(3,2)--(4,2);
\filldraw(3,0) circle[radius=2pt];
\filldraw(4,0) circle[radius=2pt];
\filldraw(3,1) circle[radius=2pt];
\filldraw(4,1) circle[radius=2pt];
\draw(3,0)--(4,0);
\draw(3,0)--(3,1);
\draw(3,1)--(4,1);
\draw(4,0)--(4,1);
\end{tikzpicture}
\caption{$C_4\cup Y_{4,2,2}$}\label{fig:y422c4}
\end{figure}
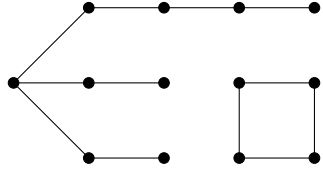

\begin{definition}\label{def:mathcaly45}
Let $n$ be a positive integer such that $n=3\times2^t - 2$ for some positive integer $t$. Define
$$\mathcal{Y}^{(4)}_n = \begin{cases}
    \mathcal{D}\left(C_4\cup Y_{4,2,2} \cup C_3 \cup C_6 \cup \displaystyle\bigcup_{i=3}^{t-1} C_{3\times 2^i}\right) &\quad\text{if }t\geq4, \\
    \mathcal{D}\left(C_4\cup Y_{4,2,2} \cup C_3 \cup C_6\right) &\quad\text{if }t=3,\\
    \varnothing &\quad\text{if }t=1\text{ or }2,
\end{cases}$$
and
$$\mathcal{Y}^{(5)}_n = \begin{cases}
    \mathcal{D}\left((K_4-e) \cup Y_{4,2,2} \cup C_3 \cup P_6 \cup \displaystyle\bigcup_{i=3}^{t-1} C_{3\times 2^i}\right) &\quad\text{if }t\geq4, \\
    \mathcal{D}\left((K_4-e) \cup Y_{4,2,2} \cup C_3 \cup P_6\right) &\quad\text{if }t=3,\\
    \varnothing &\quad\text{if }t=1\text{ or }2.
\end{cases}$$
\end{definition}

This paper will show that the cases mentioned above exhaustively list all the graphs in $\mathcal{I}(P_n)$, as stated below.

\begin{theorem}
\label{prop:pathequiv}
Let $n$ be a positive even integer such that $n+2 = 2^{t} m$ where
$t\geq 1$ and $m$ is odd. Then the independence equivalence class of
$P_{n}$, $\mathcal{I}(P_{n})$ is given by
\begin{equation*}
\mathcal{I}(P_{n})=
\begin{cases}
\mathcal{P}_{n} &\text{ if }m\neq 3,\\
\mathcal{P}_{n} \cup \mathcal{Y}^{(1)}_{n} \cup
\mathcal{Y}^{(2)}_{n} \cup \mathcal{Y}^{(3)}_{n} \cup \mathcal{Y}^{(4)}_{n}
\cup \mathcal{Y}^{(5)}_{n}&\text{ if }m=3.
\end{cases}
\end{equation*}
\end{theorem}

\begin{example}
We have already discussed $\mathcal{I}(P_{10})$. The next case is $\mathcal{I}(P_{22})$. This consists of the union of the following sets:
\begin{itemize}
    \item $\mathcal{P}_{22}$, which itself is the union of:
    \begin{itemize}
        \item $\mathcal{D}(Z_{22,0})=\{P_{22}\}$,
        \item $\mathcal{D}(Z_{22,1})=\mathcal{D}(C_{12}\cup P_{10})$, comprising two graphs ($C_{12}\cup P_{10}$ and $D_{12}\cup P_{10}$),
        \item $\mathcal{D}(Z_{22,2})=\mathcal{D}(C_{12}\cup C_6 \cup P_4)$, comprising six graphs,
        \item $\mathcal{D}(Z_{22,3})=\mathcal{D}(C_{12}\cup C_6 \cup C_3 \cup P_1)$, comprising six graphs,
    \end{itemize}
    \item $\mathcal{Y}^{(1)}_n$, which itself is the union of:
    \begin{itemize}
        \item $\mathcal{D}(Y_{3,2,1}\cup C_3 \cup C_{12})$, comprising two graphs,
        \item $\mathcal{D}(Y_{9,2,1}\cup C_3 \cup C_6)$, comprising three graphs,
    \end{itemize}
    \item $\mathcal{Y}^{(2)}_{22}=\mathcal{D}\left(P_2\cup E_{1,1} \cup C_3\cup C_{12}\right)$, comprising two graphs,
    \item $\mathcal{Y}^{(3)}_{22}=\mathcal{D}\left(P_2\cup A_{1,1} \cup C_3\cup C_{12}\right)$, comprising two graphs,
    \item $\mathcal{Y}^{(4)}_{22}=\mathcal{D}\left(C_4\cup Y_{4,2,2} \cup C_3\cup C_6\right)$, comprising six graphs,
    \item $\mathcal{Y}^{(5)}_{22}=\{(K_4-e) \cup Y_{4,2,2} \cup C_3\cup P_6\}$.
\end{itemize}
Therefore, the cardinality of $\mathcal{I}(P_{22})$ is~$31$.
\end{example}

\section{Factorisation of the Independence Polynomials of $C_n$ and $P_n$}\label{sec:factor}

In this section, we try to factorise $I(P_n,x)$ over $\mathbb{Z}[x]$ to investigate properties of the factors. Because of Proposition~\ref{prop:multi}, we know that if $G\in\mathcal{I}(P_n)$, then the independence polynomial of any connected component of $G$ would be a product of some of the factors of $I(P_n,x)$. Due to Proposition~\ref{prop:zhang}, we know that part of the problem involves factorising $I(C_n,x)$ as well, a problem that was solved in Ng~\cite{ng}. However, to make this section self-contained, we address the factorisations of $I(C_n,x)$ and $I(P_n,x)$ together.

The roots of the independence polynomials of cycle graphs and paths have been completely determined by Alikhani and Peng~\cite{alikhani}.
\begin{prop}\cite{alikhani}\label{prop:cyclicroot}
The roots of $I(C_n,x)$ are
$$c_i = -\frac{1}{2+2\cos\left(\frac{(2i-1)\pi}{n}\right)}$$
for $i = 1,2,\ldots,\left\lfloor\frac{n}{2}\right\rfloor$.
\end{prop}

\begin{prop}\cite{alikhani}\label{prop:pathroot}
The roots of $I(P_n,x)$ are
$$p_i = -\frac{1}{2+2\cos\left(\frac{2i\pi}{n+2}\right)}$$
for $i = 1,2,\ldots,\left\lfloor\frac{n+1}{2}\right\rfloor$.
\end{prop}

It can be seen that if $c_i$ is a root of $I(C_n,x)$ then it is also a root $p_{2i-1}$ of $I(P_{2n-2},x)$. The other roots of $I(P_{2n-2},x)$ are of the roots of $I(P_{n-2},x)$. This provides an alternative proof of Proposition~\ref{prop:zhang}.

The minimal polynomials of $\cos(2\pi k / n)$ have been previously determined by Lehmer~\cite{lehmer} and Watkins and Zeitlin~\cite{watkins}.

\begin{prop}\cite{lehmer}\label{prop:cosroot1}
If $\gcd(k,n)=1$ and $n\geq 3$ then the minimal polynomial of $2\cos(2\pi k / n)$ has degree $\phi(n)/2$ and leading coefficient $1$, where $\phi(n)$ is Euler's totient function.
\end{prop}
Watkins and Zeitlin~\cite{watkins} give an explicit construction for this minimal polynomial in terms of Chebychev polynomials. 
\begin{prop}\cite{watkins}\label{prop:cosroot2}
The roots of the minimal polynomial of $2\cos(2\pi / n)$ are precisely those values of $2\cos(2\pi k / n)$ for which $\gcd(k,n)=1$ and $1\leq k<n/2$.
\end{prop}

In~\cite{ng}, this was used to find the minimal polynomials of $c_i$ for odd $n$. Working along very similar methods, we can find the minimal polynomials of $c_i$ and $p_i$ for general $n$.

Since $c_i$ and $p_i$ above have similar algebraic forms, we make the following definition.
\begin{definition}\label{def:fn}
For any positive integer $n$, and any positive integer $k$ such that $1\leq k\leq n-1$, define
$$\phi_{k,n}=-\frac{1}{2+2\cos\left(\frac{k\pi}{n}\right)}.$$
\end{definition}

We now define one sequence of polynomials $f_n(x)$ for $n\geq 1$, and another sequence of polynomials $\tilde{f}_n(x)$ for odd $n\geq 3$.
\begin{definition}
For any positive integer $n$, define $f_n(x)$ to be
\begin{itemize}
    \item $f_1(x)=1$,
    \item the minimal polynomial over $\mathbb{Z}[x]$ of $\phi_{1,n}$ for $n\geq 2$.
\end{itemize}
For any odd positive integer $n\geq 3$, define $\tilde{f}_n(x)$ to be
\begin{itemize}
    \item $\tilde{f}_1(x)=1$,
    \item the minimal polynomial over $\mathbb{Z}[x]$ of $\phi_{2,n}$ for $n\geq 3$.
\end{itemize}
\end{definition}

\begin{prop}
The roots of $f_n(x)$ are $\phi_{2k-1,n}$, where $\gcd(2k-1,n)=1$ and $1\leq k \leq n/2$. The roots of $\tilde{f}_n(x)$ are $\phi_{2k,n}$, where $\gcd(k,n)=1$ and $1\leq k < n/2$.
\end{prop}

\begin{proof}
For some positive integer $n$, let $g(x)$ be the minimal polynomial of $2\cos(\pi/n)=2\cos((2\pi)/(2n))$. By Proposition~\ref{prop:cosroot2}, the roots of $g(x)$ are those values of $2\cos((2j\pi)/(2n))$, for which $\gcd(j,2n)=1$ and $1\leq j < n$. Now $\gcd(j,2n)=1$ if and only if $j$ is odd and $\gcd(j,n)=1$. Therefore, we can let $j=2k-1$ where $1\leq k\leq n/2$.

We now have
$$g(x)=\prod_{\substack{\gcd(2k-1,n)=1,\\1\leq k \leq \frac{n}{2}}} \left(x-2\cos\left(\frac{(2k-1)\pi}{n}\right)\right).$$
From Proposition~\ref{prop:cosroot1}, $g(x)$ has degree $d=\phi(2n)/2$ and leading coefficient 1. Let
$$g(x) = \sum_{t=0}^d a_t x^t.$$
Now, we can translate $g(x)$ along the $x$-axis to obtain another irreducible polynomial $h(x)$ with integer coefficients and degree $d$, whose roots are $2\cos((2k-1)\pi/n) + 2$.
$$g(x-2) = \sum_{t=0}^d a_t (x-2)^t = \sum_{t=0}^d b_t x^t = h(x).$$
Since $2\cos((2k-1)\pi/n) + 2 = -1/\phi_{2k-1,n}$, for any $k$ such that $gcd(2k-1,n)=1$ and $1\leq k \leq n/2$,
\begin{eqnarray*}
0 & = & \sum_{t=0}^d a_t \left(2\cos\left(\frac{(2k-1)\pi}{n}\right)\right)^t \\
& = & \sum_{t=0}^d b_t \left(2\cos\left(\frac{(2k-1)\pi}{n}\right) + 2\right)^t \\
& = & \sum_{t=0}^d b_t \left(-\frac{1}{\phi_{2k-1,n}}\right)^t \\
& = & \left(-\frac{1}{\phi_{2k-1,n}}\right)^d \sum_{t=0}^d b_t (-\phi_{2k-1,n})^{d-t}.
\end{eqnarray*}
and therefore the $\phi_{2k-1,n}$ are roots of a polynomial
$$f(x) = \sum_{t=0}^d b_t (-x)^{d-t}.$$
Since $h(x)$ is irreducible, $f(x)$ is irreducible as well, and has degree $d$ and integer coefficients.

The roots of $f(x)$ are precisely the values of $\phi_{k,n}$ where $\gcd(2k-1,n)=1$ and $1\leq k \leq \frac{n}{2}$.

Similarly, for some odd positive integer $n\geq 3$, let $\tilde{g}(x)$ be the minimal polynomial of $2\cos(2\pi/n)$. By Proposition~\ref{prop:cosroot2}, the roots of $\tilde{g}(x)$ are $2\cos(2k\pi/n)$, for which $\gcd(k,n)=1$ and $1\leq k < n/2$.

We now have
$$\tilde{g}(x)=\prod_{\substack{\gcd(k,n)=1,\\1\leq k < \frac{n}{2}}} \left(x-2\cos\left(\frac{2k\pi}{n}\right)\right).$$
From Proposition~\ref{prop:cosroot1}, $g(x)$ has degree $\tilde{d}=\phi(n)/2$ and leading coefficient 1. Let
$$\tilde{g}(x) = \sum_{t=0}^{\tilde{d}} \tilde{a}_t x^t.$$
Now, we can translate $\tilde{g}(x)$ along the $x$-axis to obtain another irreducible polynomial $\tilde{h}(x)$ with integer coefficients and degree $\tilde{d}$, whose roots are $2\cos(2k\pi/n) + 2$.
$$\tilde{g}(x-2) = \sum_{t=0}^{\tilde{d}} \tilde{a}_t (x-2)^t = \sum_{t=0}^{\tilde{d}} \tilde{b}_t x^t = h(x).$$
Since $2\cos(2k\pi/n) + 2 = -1/\phi_{2k,n}$,
\begin{eqnarray*}
0 & = & \sum_{t=0}^{\tilde{d}} \tilde{a}_t \left(2\cos\left(\frac{2k\pi}{n}\right)\right)^t \\
& = & \sum_{t=0}^{\tilde{d}} \tilde{b}_t \left(2\cos\left(\frac{2k\pi}{n}\right) + 2\right)^t \\
& = & \sum_{t=0}^{\tilde{d}} \tilde{b}_t \left(-\frac{1}{\phi_{2k,n}}\right)^t \\
& = & \left(-\frac{1}{\phi_{2k,n}}\right)^{\tilde{d}} \sum_{t=0}^{\tilde{d}} \tilde{b}_t (-\phi_{2k,n})^{\tilde{d}-t}.
\end{eqnarray*}
and therefore the $\phi_{2k,n}$ are roots of a polynomial
$$ \tilde{f}(x) = \sum_{t=0}^{\tilde{d}} \tilde{b}_t (-x)^{\tilde{d}-t}. $$
Since $\tilde{h}(x)$ is irreducible, $\tilde{f}(x)$ is irreducible as well, and has degree $\tilde{d}$ and integer coefficients.

The roots of $\tilde{f}(x)$ are precisely the values of $\phi_{2k,n}$ where $\gcd(k,n)=1$ and $1\leq k < \frac{n}{2}$.
\end{proof}

\begin{corollary}\label{cor:unique}
For each $\phi_{k,n}$, where $1\leq k\leq n-1$, the monomial $(x-\phi_{k,n})$ is a factor of exactly one of the $f_n(x)$ or one of the $\tilde{f}_n(x)$ (but not both).
\end{corollary}
\begin{proof}
Let $d=\gcd(k,m)$, and $k^\prime=k/d$ and $m^\prime = m/d$. Then $k^\prime/m^\prime$ is the unique way to represent $k/m$ in lowest terms, and $\phi_{k,m}=\phi_{k^\prime,m^\prime}$. Now $(x-\phi_{k,m})\vert f_{m^\prime}(x)$ iff $k^\prime$ is odd and $(x-\phi_{k,m})\vert \tilde{f}_{m^\prime}(x)$ iff $k^\prime$ is even.
\end{proof}

\begin{prop}\label{prop:cyclicfactor}
For $n>1$, let $n=2^t m$ where $k$ is a non-negative integer and $m$ is an odd positive integer. Then
$$I(C_n,x)=\prod_{r\vert m} f_{2^t r}(x).$$
\end{prop}

\begin{proof}
From Proposition~\ref{prop:cyclicroot}, it is clear that
$$I(C_n,x)=\prod_{k=1}^{\left\lfloor\frac{n}{2}\right\rfloor} \left(x-\phi_{2k-1,n}\right).$$

For all $r\vert m$, the factors of $f_{2^t r}(x)$ are $(x-\phi_{2j-1,2^t r})$, where $\gcd(2j-1,2^t r)=1$. But $d=m/r$ is an odd integer, and $\phi_{2j-1,2^t r}=\phi_{(2j-1)d,2^t m}=\phi_{(2k-1), n}$ for some value of $k$. Hence, $f_{2^t r}(x) \vert I(C_n,x)$. From Corollary~\ref{cor:unique}, none of the $f_{2^t r}(x)$ have common factors, and so
$$\left.\prod_{r\vert m} f_{2^t r}(x)\right\vert I(C_n,x).$$

For each value of $2k-1$, let $d_k=\gcd(2k-1,n)$. Then
$\phi_{2k-1,n} = \phi_{(2k-1)/d_k,n/d_k}$, and $\gcd((2k-1)/d_k,n/d_k)=1$. Therefore, $(x-\phi_{2k-1,n})$ is a factor of $f_{n/d_k}(x)$. Since $d_k\vert(2k-1)$, $d_k$ is odd, and hence $n/d_k = 2^t r$ for some $r\vert m$. Therefore, $(x-\phi_{2k-1,n})\vert f_{2^t r}(x)$ for some $r\vert m$, and so
$$I(C_n,x)\left\vert\prod_{r\vert m} f_{2^t r}(x)\right..$$
\end{proof}

\begin{corollary}\label{cor:cyclicfactor}
For two cycle graphs $C_k$ and $C_n$, $I(C_k,x)\vert I(C_n,x)$ if and only if $n/k$ is an odd number.
\end{corollary}

\begin{prop}\label{prop:pathfactor}
For positive odd $n$,
$$I(P_{n},x)=\prod_{r\vert (n+2)} \tilde{f}_r(x).$$
For positive even $n$, let $n+2=2^t m$ where $t$ and $m$ are positive integers and $m$ is odd. Then
$$I(P_{n},x) = \prod_{r\vert 2^{t-1}m} f_r(x) \prod_{s\vert m} \tilde{f}_s(x).$$
\end{prop}
\begin{proof}
From Proposition~\ref{prop:pathroot}, it is clear that
$$I(P_{n},x)=\prod_{k=1}^{\left\lfloor\frac{n+1}{2}\right\rfloor} \left(x-\phi_{2k,n+2}\right).$$

If $n$ is odd, then $\gcd(2k,n+2)=1$ iff $\gcd(k,n+2)=1$. For all $r\vert (n+2)$, where $r>1$, the factors of $\tilde{f}_r(x)$ are $(x-\phi_{2k,r})$, where $\gcd(k,r)=1$. From Corollary~\ref{cor:unique}, none of the $\tilde{f}_r (x)$ have common factors, and so
$$\left.\prod_{r\vert (n+2)} \tilde{f}_r(x)\right\vert I(P_{n},x).$$

For each value of $2k$, let $d_k=\gcd(2k,n+2)$, and let $(n+2)/d_k = r$. Then
$\phi_{2k,n+2} = \phi_{2k/d_k,r}$, and $\gcd(2k/d_k,r)=1$. Therefore, $(x-\phi_{2k,n+2})$ is a factor of $f_r (x)$ for some $r\vert (n+2)$, and so
$$I(P_{n},x)\left\vert\prod_{r\vert (n+2)} \tilde{f}_r(x)\right..$$

Therefore,
$$I(P_{n},x)=\prod_{r\vert (n+2)} \tilde{f}_r(x)$$
for positive odd $n$.

For positive even $n$, we repeatedly apply Proposition~\ref{prop:zhang}.
\begin{eqnarray*}
I(P_{2^t m -2},x) & = & I(P_{2^{t-1} m -2},x)I(C_{2^{t-1} m}, x) \\
& = & I(P_{2^{t-2} m -2},x) I(C_{2^{t-2} m}, x) I(C_{2^{t-1} m}, x) \\
& \vdots & \\
& = & I(P_{m-2},x) \prod_{i=0}^{t-1}  I(C_{2^i m}, x)
\end{eqnarray*}
Since $m-2$ is odd,
$$I(P_{m-2},x)=\prod_{s\vert m} \tilde{f}_s(x),$$
Also, for each $i$,
$$I(C_{2^i m}, x)=\prod_{r\vert m} f_{2^i r}(x),$$
so that
$$\prod_{i=0}^{t-1}  I(C_{2^i m}, x)=\prod_{i=0}^{t-1}\prod_{r\vert m} f_{2^i r}(x).$$
But as $i$ ranges from $0$ to $t-1$ and $r$ ranges over all factors of $m$, $2^i r$ would range over all factors of $2^{t-1}m$. Hence,
$$I(P_n,x) = \prod_{r\vert 2^{t-1}m} f_r(x) \prod_{s\vert m} \tilde{f}_s(x).$$
\end{proof}

In~\cite{ng}, examples are given for factorising $I(C_n,x)$ for $n=3,5,9,15$. Here, we will illustrate the factorisation of some even cycle graphs and paths.

\begin{example}
We have
$$f_2(x) = I(P_2,x) = 1 + 2x$$
and
$$f_3(x) = I(C_3,x) = 1 + 3x.$$
Both of these are irreducible polynomials.
\end{example}
\begin{example}
We have
$$\tilde{f}_3(x) = I(P_1,x) = 1+x.$$
\end{example}
\begin{example}
We have
\begin{eqnarray*}
I(C_6,x) & = & 1 + 6x + 9x^2 + 2x^3\\
& = & \underbrace{(1 + 2x)}_{f_2(x)}\underbrace{(1 + 4x + x^2)}_{f_6(x)}.
\end{eqnarray*}
It so happens that
$$I(K_4 - e, x) = 1 + 4x + x^2 = f_6(x),$$
which explains why $K_4-e\cup P_2$ is independence equivalent to $C_6$.
\end{example}
\begin{example}
As mentioned in the previous section, $\mathcal{I}(P_{10})$ contains ten graphs~\cite{beaton}. We can determine which graphs are in $\mathcal{I}(P_{10})$ by another approach, namely, that of factorising $I(P_{10},x)$.
\begin{eqnarray*}
I(P_{10},x) & = & 1 + 10x + 36x^2 + 56x^3 + 35x^4 + 6x^5\\
& = & \underbrace{(1 + 2x)}_{f_2(x)}\underbrace{(1 + 3x)}_{f_3(x)}\underbrace{(1 + 4x + x^2)}_{f_6(x)}\underbrace{(1 + x)}_{\tilde{f}_3(x)}.
\end{eqnarray*}
Some of the graphs mentioned in the previous section are listed below with their independence polynomial.
\begin{center}
\begin{tabular}{c|c}
Graph $G$ & Independence polynomial $I(G,x)$ \\
\hline
$P_{10}$ & $f_2(x)f_3(x)f_6(x)\tilde{f}_3(x)$ \\
$Y_{3,2,1}$ (see Figure~\ref{fig:g01}) & $f_2(x)f_6(x)\tilde{f}_3(x)$ \\
$C_6$, $D_6$ & $f_2(x)f_6(x)$ \\
$E_{1,1}$ (see Figure~\ref{fig:g01}), $A_{1,1}$ (see Figure~\ref{fig:g12}) & $f_6(x)\tilde{f}_3(x)$ \\
$K_4-e$ & $f_6(x)$ \\
$C_3$ & $f_3(x)$ \\
$P_2$ & $f_2(x)$ \\
$P_1$ & $\tilde{f}_3(x)$
\end{tabular}
\end{center}
We can see which of these graphs' independence polynomials can be multiplied together to give $I(P_{10},x)$.

\end{example}

\section{Structure of a graph which is in $\mathcal{I}(P_n)$}

\begin{definition}
For any graph~$G$, let $n_G(C_3)$ be the number of triangles in $G$, and $g_i$ denote the number of vertices in $G$ of degree~$i$.
\end{definition}

\begin{prop}\label{prop:degreesum}
If $G$ is a graph such that $G\in\mathcal{I}(P_n)$, then
\begin{enumerate}[label=(\roman*)]
    \item $\displaystyle\sum_{i=0}^{n-1}g_i = n$,
    \item $\displaystyle\sum_{i=1}^{n-1}i\cdot g_i = 2n-2$,
    \item $\displaystyle\sum_{i=2}^{n-1}\binom{i}{2}g_i = n - 2 + n_G(C_3)$,
    \item $n_G(C_3) = g_0 + \displaystyle\sum_{i=3}^{n-1}\binom{i-1}{2}g_i$.
\end{enumerate}
\end{prop}

The proof of this borrows very heavily from Beaton, Brown and Cameron~\cite{beaton}.

\begin{proof}
Note that $i_1(G)=i_1(P_n)=n$ since the number of independent sets of cardinality~1 in a graph is simply the number of vertices of a graph. Furthermore,
$$i_2(G)=\binom{|V(G)|}{2}-|E(G)|,$$
and since $i_2(G)=i_2(P_n)$, $G$ and $P_n$ must have the same number of edges as well. Now the left side of~(i) is $|V(G)|=n$. The left side of~(ii) counts the number of ordered pairs $(v,e)$ where $v$ is a vertex of $G$ and $e$ is an edge incident to $v$. In other words, it is $2|E(G)|=2|E(P_n)|=2n-2$. The identity~(iii) follows from Theorem~3.2 of Beaton, Brown and Cameron~\cite{beaton} and the fact that $i_3(G)=i_3(P_n)$. Finally,~(iv) follows from adding~(i) and~(iii) and subtracting~(ii).
\end{proof}

Since the independence polynomial of $P_n$ does not have repeated roots, at most one of the connected components can be an isolated vertex. In other words, $g_0\leq 1$. Furthermore, at most one of the connected components can be $C_3$. If any of the other connected components contains a triangle, then at least one of the vertices of the triangle must have degree~3 or larger.

\begin{prop}\label{prop:maxdegree3}
If $G$ is a graph such that $G\in\mathcal{I}(P_n)$, then the maximum degree $\Delta(G)$ of $G$ is at most~$3$ and $n_G(C_3)=g_0 + g_3$.
\end{prop}

\begin{proof}
From Proposition~\ref{prop:degreesum}(iv),
$$n_G(C_3)=g_0+g_3+3g_4+6g_5+\ldots+\binom{n-2}{2}g_{n-1}.$$
Let $h=g_4+\ldots+g_{n-1}>0$. Then $n_G(C_3)\geq g_0+g_3+3h$. Since every triangle, except at most one, has a vertex of degree~3 or larger, we see that $g_3+h$ must be at least as large as the number of triangles less one, that is, $g_3+h\geq n_G(C_3)-1$. Therefore, we have $g_3+h\geq g_0+g_3+3h-1$, which simplifies to $2h\leq 1-g_0\leq 1$. Since $h$ is a nonnegative integer, the only possibility is that $h=0$, so $g_4=\ldots=g_{n-1}=0$ and the maximum degree of $G$ is~$3$. This also implies that $n_G(C_3)=g_0+g_3$.
\end{proof}

Now, since Wingard~\cite{wingard} has already shown that if $G$ is a connected graph in $\mathcal{I}(P_n)$, then $G$ must be isomorphic to $P_n$, we can assume that $G$ contains connected components $G_1,\ldots,G_q$, where $q\geq 2$. For each connected component $G_i$, let $g_{i,j}$ be the number of vertices of degree~$j$ in $G_i$.

\begin{prop}\label{prop:maxdegree3verts}
Let $G_i$ be a connected graph. If $\Delta(G_i)=3$ and $G_i$ is not $K_4$, then $g_{i,3}\geq 2n_{G_i}(C_3)-2$.
\end{prop}
\begin{proof}
Since $\Delta(G_i)=3$, $G_i$ has at least 4 vertices. If $G_i$ is not $K_4$ then three triangles cannot share a vertex, otherwise the edges of the triangles which are not incident to that vertex would form a fourth triangle and create a $K_4$ induced subgraph, which would be the entirety of $G_i$ as the maximum degree of $G_i$ is~3.

If two triangles share a vertex then they also share an edge $e$, and both vertices incident to the shared edge would have degree~3. None of the remaining edges of the triangles (two on each triangle) can be shared with a third triangle, as this would cause one of the vertices incident to $e$ to be of degree~4. Therefore, every triangle has at least one edge that is not shared with any other triangle.

Therefore, the triangles can form a subset of a cycle basis of $G_i$, so
$$n_{G_i}(C_3)\leq |E(G_i)|-|V(G_i)|+1$$
as $G_i$ is connected (see Chapter~4.6 of~\cite{gross}). Since
$$|V(G_i)|=g_{i,1}+g_{i,2}+g_{i,3}$$
and
$$|E(G_i)|=\frac{1}{2}\left(g_{i,1}+2g_{i,2}+3g_{i,3}\right),$$
we have
$$n_{G_i}(C_3)\leq \frac{1}{2}\left(g_{i,3}-g_{i,1}\right)+1\leq\frac{g_{i,3}}{2}+1$$
and the result follows.
\end{proof}

\begin{corollary}\label{cor:maxdegree3verts1}
Let $G_i$ be a connected graph. If $\Delta(G_i)=3$ and $n_{G_i}(C_3)\geq 2$, then $g_{i,3} \geq n_{G_i}(C_3)$.
\end{corollary}
\begin{proof}
If $G_i = K_4$ then $g_{i,3} = 4 = n_{G_i}(C_3)$. Otherwise, $2n_{G_i}(C_3)-2\geq n_{G_i}(C_3)$ for $n_{G_i}(C_3)\geq 2$.
\end{proof}

\begin{corollary}\label{cor:maxdegree3verts2}
Let $G_i$ be a connected graph. If $\Delta(G_i)=3$ and $n_{G_i}(C_3)\geq 4$ then either $G_i$ is $K_4$, or $g_{i,3} \geq n_{G_i}(C_3)+2$.
\end{corollary}
\begin{proof}
If $n_{G_i}(C_3)\geq 4$ then 
$2n_{G_i}(C_3)-2\geq n_{G_i}(C_3)+2$.
\end{proof}

\begin{prop}\label{prop:maxtriangles}
If $G$ is a graph such that $I(G,x)=I(P_n,x)$, then for each connected component $G_i$ of $G$, we have $g_{i,3}\leq n_{G_i}(C_3)+1$, and each $G_i$ has at most three triangles.
\end{prop}
\begin{proof}
From Proposition~\ref{prop:maxdegree3}, we know that $n_G(C_3)=g_0+g_3$, and since $I(G,x)$ has no repeated roots, $G$ can have at most one isolated vertex, so $0\leq g_0\leq 1$, and, therefore,
\begin{equation}\label{eq:maxdegree3-1}
g_3\leq n_G(C_3) \leq g_3+1
\end{equation}
Consider a connected component $G_i$ of $G$. If $n_{G_i}(C_3)=0$ then $g_{i,3}\geq n_{G_i}(C_3)$ trivially. If $n_{G_i}(C_3)=1$, then either $G_i=K_3$, in which case $g_{i,3}=0$, or $G_i$ must have at least one vertex of degree~3, so $g_{i,3}\geq 1$. Note that since $I(G,x)$ has no repeated roots, at most one $G_i$ can be $K_3$. Finally, from Corollary~\ref{cor:maxdegree3verts1}, if $n_{G_i}(C_3)\geq 2$ then $g_{i,3}\geq n_{G_i}(C_3)$. In summary,
\begin{equation}\label{eq:maxdegree3-2}
g_{i,3}\geq n_{G_i}(C_3)
\end{equation}
for all values of $i$ except at most one.

On the other hand, $g_3 =\sum_{i=1}^q g_{i,3}$ and $n_G(C_3) =\sum_{i=1}^q n_{G_i}(C_3)$.

\textbf{Case 1.} If none of the $G_i$ are $K_3$ then summing the inequality (\ref{eq:maxdegree3-2}) from $1$ to $r$ gives us $g_3 \geq n_G(C_3)$. Comparing this with (\ref{eq:maxdegree3-1}) means that $g_3 = n_G(C_3)$ and, in fact, $g_{i,3} = n_{G_i}(C_3)$ for each $i$.

\textbf{Case 2.} If, without loss of generality, $G_1=K_3$, then summing the inequality (\ref{eq:maxdegree3-2}) from $2$ to $r$, along with $g_{1,3}=0, n_{G_1}(C_3)=1$ gives us $g_3+1\geq n_G(C_3)$. Comparing this with (\ref{eq:maxdegree3-1}) means that $g_3+1 = n_G(C_3)$, and so, without loss of generality, $g_{r,3}=n_{G_r}(C_3)+1$ and $g_{i,3} = n_{G_i}(C_3)$ for each $i\in\{2,\ldots,q-1\}$.

In both cases, $g_{i,3} \leq n_{G_i}(C_3)+1$ for all $i$.

None of the $G_i$ can be $K_4$ as $I(K_4,x)=1+4x$ and its independence root is $-\frac{1}{4}$, which is not an independence root of any path (Proposition~\ref{prop:pathroot}). Using Corollary~\ref{cor:maxdegree3verts2}, we find that if $n_{G_i}(C_3)\geq 4$ for any $G_i$, then $g_{i,3} \geq n_{G_i}(C_3)+2$, a contradiction. Hence there are at most three triangles in each $G_i$.
\end{proof}

We can now catalogue a list of graph families based on the number of triangles and number of vertices of degree~3.
\begin{center}
\begin{tabular}{c|c|c}
$n_{G_i}(C_3)$ & $g_{i,3}$ & Candidates for $G_i$ \\
\hline
0 & 0 & $P_m$ ($m\geq 1$) and $C_m$ ($m\geq 4$) \\
0 & 1 & See Figure~\ref{fig:g01} \\
1 & 0 & $C_3$ \\
1 & 1 & $D_m$ ($m \geq 4$) \\
1 & 2 & See Figure~\ref{fig:g12} \\
2 & 2 & See Figure~\ref{fig:g22} \\
2 & 3 & See Figure~\ref{fig:g23} \\
3 & 4 & See Figure~\ref{fig:g34} 
\end{tabular}
\end{center}

In the next section, we will narrow down this list of graphs further to determine which can possibly be connected components of a graph $G\in\mathcal{I}(P_n)$.

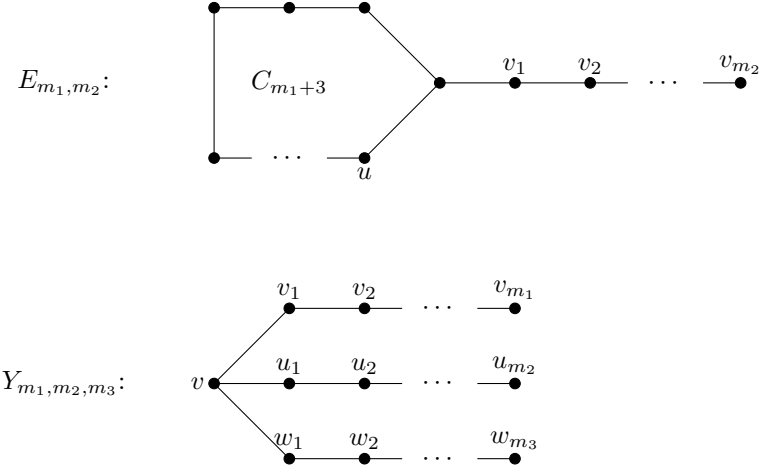
\begin{figure}
\centering
\begin{tikzpicture}
\node at (-2,9) {$E_{m_1,m_2}$:};
\filldraw(0,10) circle[radius=2pt];
\filldraw(1,10) circle[radius=2pt];
\filldraw(2,10) circle[radius=2pt];
\filldraw(0,8) circle[radius=2pt];
\node at (1,8) {$\cdots$};
\node at (1,9) {$C_{m_1+3}$};
\filldraw(2,8) circle[radius=2pt]node[below]{$u$};
\filldraw(3,9) circle[radius=2pt];
\filldraw(4,9) circle[radius=2pt]node[above]{$v_1$};
\filldraw(5,9) circle[radius=2pt]node[above]{$v_2$};
\node at (6,9) {$\cdots$};
\filldraw(7,9) circle[radius=2pt]node[above]{$v_{m_2}$};
\draw(0,10)--(1,10);
\draw(1,10)--(2,10);
\draw(0,10)--(0,8);
\draw(0,8)--(0.5,8);
\draw(1.5,8)--(2,8);
\draw(2,10)--(3,9);
\draw(2,8)--(3,9);
\draw(3,9)--(4,9);
\draw(4,9)--(5,9);
\draw(5,9)--(5.5,9);
\draw(6.5,9)--(7,9);

\node at (-2,5) {$Y_{m_1,m_2,m_3}$:};
\filldraw(0,5) circle[radius=2pt]node[left]{$v$};
\filldraw(1,6) circle[radius=2pt]node[above]{$v_1$};
\filldraw(2,6) circle[radius=2pt]node[above]{$v_2$};
\node at (3,6) {$\cdots$};
\filldraw(4,6) circle[radius=2pt]node[above]{$v_{m_1}$};
\filldraw(1,5) circle[radius=2pt]node[above]{$u_1$};
\filldraw(2,5) circle[radius=2pt]node[above]{$u_2$};
\node at (3,5) {$\cdots$};
\filldraw(4,5) circle[radius=2pt]node[above]{$u_{m_2}$};
\filldraw(1,4) circle[radius=2pt]node[above]{$w_1$};
\filldraw(2,4) circle[radius=2pt]node[above]{$w_2$};
\node at (3,4) {$\cdots$};
\filldraw(4,4) circle[radius=2pt]node[above]{$w_{m_3}$};
\draw(0,5)--(1,6);
\draw(1,6)--(2,6);
\draw(2,6)--(2.5,6);
\draw(3.5,6)--(4,6);
\draw(0,5)--(1,5);
\draw(1,5)--(2,5);
\draw(2,5)--(2.5,5);
\draw(3.5,5)--(4,5);
\draw(0,5)--(1,4);
\draw(1,4)--(2,4);
\draw(2,4)--(2.5,4);
\draw(3.5,4)--(4,4);
\end{tikzpicture}
\caption{Graphs with $n_{G_i}(C_3)=0$ and $g_{i,3}=1$. All parameters are positive integers.}\label{fig:g01}
\end{figure}
\begin{figure}
\centering
\begin{tikzpicture}
\node at (-2,5) {$A_{m_1,m_2}$:};
\filldraw(0,5) circle[radius=2pt]node[left]{$v$};
\filldraw(1,6) circle[radius=2pt];
\filldraw(2,6) circle[radius=2pt]node[above]{$v_1$};
\filldraw(3,6) circle[radius=2pt]node[above]{$v_2$};
\node at (4,6) {$\cdots$};
\filldraw(5,6) circle[radius=2pt]node[above]{$v_{m_1}$};
\filldraw(1,4) circle[radius=2pt];
\filldraw(2,4) circle[radius=2pt]node[below]{$u_1$};
\filldraw(3,4) circle[radius=2pt]node[below]{$u_2$};
\node at (4,4) {$\cdots$};
\filldraw(5,4) circle[radius=2pt]node[below]{$u_{m_2}$};
\draw(0,5)--(1,6);
\draw(0,5)--(1,4);
\draw(1,6)--(2,6);
\draw(2,6)--(3,6);
\draw(3,6)--(3.5,6);
\draw(4.5,6)--(5,6);
\draw(1,6)--(1,4);
\draw(1,4)--(2,4);
\draw(2,4)--(3,4);
\draw(3,4)--(3.5,4);
\draw(4.5,4)--(5,4);

\node at (-2,1) {$B_{m_1,m_2,m_3}$:};
\filldraw(0,0) circle[radius=2pt];
\filldraw(0,2) circle[radius=2pt];
\filldraw(1,1) circle[radius=2pt]node[above]{$v$};
\filldraw(2,1) circle[radius=2pt]node[above]{$v_1$};
\node at (3,1) {$\cdots$};
\filldraw(4,1) circle[radius=2pt]node[above]{$v_{m_1}$};
\filldraw(5,1) circle[radius=2pt]node[right]{$u$};
\filldraw(5,2) circle[radius=2pt]node[above]{$u_1$};
\filldraw(6,2) circle[radius=2pt]node[above]{$u_2$};
\node at (7,2) {$\cdots$};
\filldraw(8,2) circle[radius=2pt]node[above]{$u_{m_2}$};
\filldraw(5,0) circle[radius=2pt]node[below]{$w_1$};
\filldraw(6,0) circle[radius=2pt]node[below]{$w_2$};
\node at (7,0) {$\cdots$};
\filldraw(8,0) circle[radius=2pt]node[below]{$w_{m_3}$};
\draw(0,0)--(0,2);
\draw(0,0)--(1,1);
\draw(0,2)--(1,1);
\draw(1,1)--(2,1);
\draw(2,1)--(2.5,1);
\draw(3.5,1)--(4,1);
\draw(4,1)--(5,1);
\draw(5,1)--(5,2);
\draw(5,2)--(6,2);
\draw(6,2)--(6.5,2);
\draw(7.5,2)--(8,2);
\draw(5,1)--(5,0);
\draw(5,0)--(6,0);
\draw(6,0)--(6.5,0);
\draw(7.5,0)--(8,0);

\node at (-2,-3) {$F^{(1)}_{m_1,m_2}$:};
\filldraw(0,-2) circle[radius=2pt];
\filldraw(0,-4) circle[radius=2pt];
\filldraw(1,-3) circle[radius=2pt]node[above]{$v$};
\filldraw(2,-3) circle[radius=2pt]node[above]{$v_1$};
\node at (3,-3) {$\cdots$};
\filldraw(4,-3) circle[radius=2pt]node[above]{$v_{m_1}$};
\filldraw(5,-3) circle[radius=2pt]node[above]{$u$};
\filldraw(6,-2) circle[radius=2pt];
\filldraw(7,-2) circle[radius=2pt];
\filldraw(8,-2) circle[radius=2pt];
\filldraw(8,-4) circle[radius=2pt];
\filldraw(6,-4) circle[radius=2pt];
\node at (7,-3) {$C_{m_2+3}$};
\node at (7,-4) {$\cdots$};
\draw(0,-2)--(0,-4);
\draw(0,-2)--(1,-3);
\draw(0,-4)--(1,-3);
\draw(1,-3)--(2,-3);
\draw(2,-3)--(2.5,-3);
\draw(3.5,-3)--(4,-3);
\draw(4,-3)--(5,-3);
\draw(5,-3)--(6,-2);
\draw(6,-2)--(7,-2);
\draw(7,-2)--(8,-2);
\draw(8,-2)--(8,-4);
\draw(5,-3)--(6,-4);
\draw(6,-4)--(6.5,-4);
\draw(7.5,-4)--(8,-4);

\node at (-2,-7) {$F^{(2)}_m$:};
\filldraw(0,-7) circle[radius=2pt]node[left]{$v$};
\filldraw(1,-6) circle[radius=2pt];
\filldraw(1,-8) circle[radius=2pt];
\filldraw(2,-6) circle[radius=2pt];
\filldraw(3,-6) circle[radius=2pt];
\filldraw(3,-8) circle[radius=2pt];
\node at (2,-8) {$\cdots$};
\node at (2,-7) {$C_{m+3}$};
\draw(0,-7)--(1,-6);
\draw(0,-7)--(1,-8);
\draw(1,-6)--(1,-8);
\draw(1,-6)--(2,-6);
\draw(2,-6)--(3,-6);
\draw(3,-6)--(3,-8);
\draw(1,-8)--(1.5,-8);
\draw(2.5,-8)--(3,-8);
\end{tikzpicture}
\caption{Graphs with $n_{G_i}(C_3)=1$ and $g_{i,3}=2$. For $B_{m_1,m_2,m_3}$ and $F^{(1)}_{m_1,m_2}$, if $m_1=0$ then the two vertices of degree~3 are adjacent. All other parameters are positive integers.}\label{fig:g12}
\end{figure}
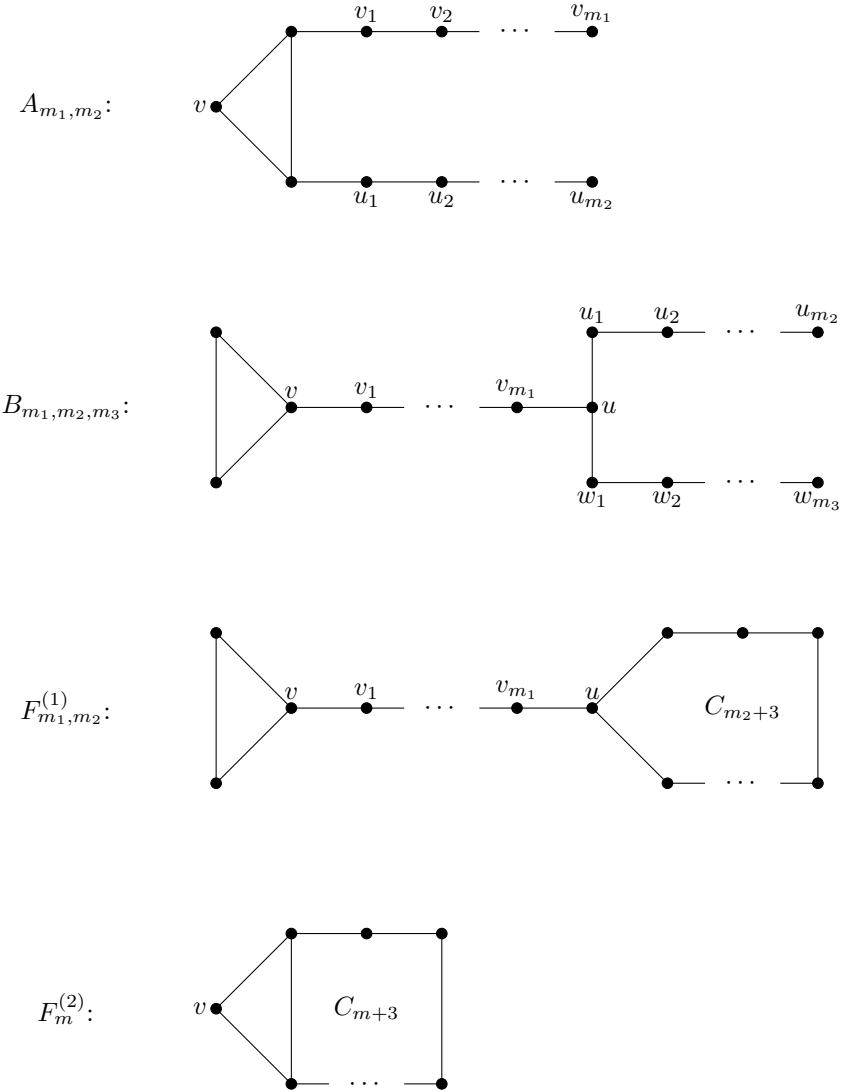
\begin{figure}
\centering
\begin{tikzpicture}
\node at (-2,9) {$K_4-e$:};
\filldraw(0,9) circle[radius=2pt];
\filldraw(1,10) circle[radius=2pt];
\filldraw(1,8) circle[radius=2pt];
\filldraw(2,9) circle[radius=2pt];
\draw(0,9)--(1,10);
\draw(0,9)--(1,8);
\draw(1,10)--(2,9);
\draw(1,8)--(2,9);
\draw(1,10)--(1,8);

\node at (-2,5) {$F^{(3)}_m$:};
\filldraw(0,6) circle[radius=2pt];
\filldraw(0,4) circle[radius=2pt];
\filldraw(1,5) circle[radius=2pt];
\filldraw(2,5) circle[radius=2pt]node[above]{$v_1$};
\node at (3,5) {$\cdots$};
\filldraw(4,5) circle[radius=2pt]node[above]{$v_m$};
\filldraw(5,5) circle[radius=2pt];
\filldraw(6,6) circle[radius=2pt];
\filldraw(6,4) circle[radius=2pt]node[right]{$v$};
\draw(0,6)--(0,4);
\draw(0,6)--(1,5);
\draw(0,4)--(1,5);
\draw(1,5)--(2,5);
\draw(2,5)--(2.5,5);
\draw(3.5,5)--(4,5);
\draw(4,5)--(5,5);
\draw(5,5)--(6,6);
\draw(5,5)--(6,4);
\draw(6,6)--(6,4);
\end{tikzpicture}
\caption{Graphs with $n_{G_i}(C_3)=2$ and $g_{i,3}=2$. For $F^{(3)}_m$, if $m=0$ then the two vertices of degree~3 are adjacent.}\label{fig:g22}
\end{figure}
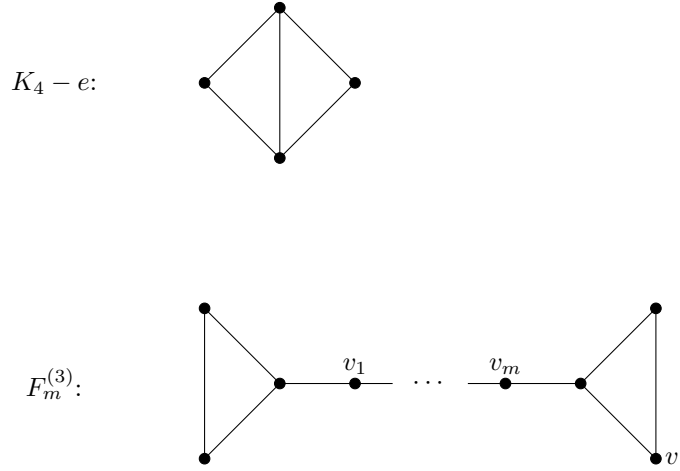
\begin{figure}
\centering
\begin{tikzpicture}
\node at (0,10.5) {$F^{(4)}_m$:};
\filldraw(0,9) circle[radius=2pt]node[left]{$v$};
\filldraw(1,10) circle[radius=2pt];
\filldraw(1,8) circle[radius=2pt];
\filldraw(2,9) circle[radius=2pt];
\filldraw(3,9) circle[radius=2pt]node[above]{$v_1$};
\node at (4,9) {$\cdots$};
\filldraw(5,9) circle[radius=2pt]node[above]{$v_m$};
\draw(0,9)--(1,10);
\draw(0,9)--(1,8);
\draw(1,10)--(2,9);
\draw(1,8)--(2,9);
\draw(1,10)--(1,8);
\draw(2,9)--(3,9);
\draw(3,9)--(3.5,9);
\draw(4.5,9)--(5,9);

\node at (0,6.5) {$F^{(5)}_{m_1,m_2}$:};
\filldraw(0,6) circle[radius=2pt];
\filldraw(0,4) circle[radius=2pt];
\filldraw(1,5) circle[radius=2pt]node[above]{$v$};
\filldraw(2,5) circle[radius=2pt]node[above]{$v_1$};
\node at (3,5) {$\cdots$};
\filldraw(4,5) circle[radius=2pt]node[above]{$v_{m_1}$};
\filldraw(5,5) circle[radius=2pt]node[above]{$u$};
\filldraw(6,6) circle[radius=2pt];
\filldraw(6,4) circle[radius=2pt];
\filldraw(7,6) circle[radius=2pt]node[below]{$u_1$};
\node at (8,6) {$\cdots$};
\filldraw(9,6) circle[radius=2pt]node[below]{$u_{m_2}$};
\draw(0,6)--(0,4);
\draw(0,6)--(1,5);
\draw(0,4)--(1,5);
\draw(1,5)--(2,5);
\draw(2,5)--(2.5,5);
\draw(3.5,5)--(4,5);
\draw(4,5)--(5,5);
\draw(5,5)--(6,6);
\draw(5,5)--(6,4);
\draw(6,6)--(6,4);
\draw(6,6)--(7,6);
\draw(7,6)--(7.5,6);
\draw(8.5,6)--(9,6);

\node at (0,2.5) {$F^{(6)}_{m_1,m_2,m_3}$:};
\filldraw(0,2) circle[radius=2pt];
\filldraw(0,0) circle[radius=2pt];
\filldraw(1,1) circle[radius=2pt]node[above]{$v$};
\filldraw(2,1) circle[radius=2pt]node[above]{$v_1$};
\node at (3,1) {$\cdots$};
\filldraw(4,1) circle[radius=2pt]node[above]{$v_{m_1}$};
\filldraw(5,1) circle[radius=2pt]node[above]{$u$};
\filldraw(6,1) circle[radius=2pt]node[above]{$u_1$};
\node at (7,1) {$\cdots$};
\filldraw(8,1) circle[radius=2pt]node[above]{$u_{m_2}$};
\filldraw(9,1) circle[radius=2pt]node[above]{$w$};
\filldraw(10,0) circle[radius=2pt];
\filldraw(10,2) circle[radius=2pt];
\filldraw(5,0) circle[radius=2pt]node[below]{$w_1$};
\node at (6,0) {$\cdots$};
\filldraw(7,0) circle[radius=2pt]node[below]{$w_{m_3}$};
\draw(0,2)--(0,0);
\draw(0,2)--(1,1);
\draw(0,0)--(1,1);
\draw(1,1)--(2,1);
\draw(2,1)--(2.5,1);
\draw(3.5,1)--(4,1);
\draw(4,1)--(5,1);
\draw(5,1)--(6,1);
\draw(6,1)--(6.5,1);
\draw(7.5,1)--(8,1);
\draw(8,1)--(9,1);
\draw(9,1)--(10,2);
\draw(9,1)--(10,0);
\draw(10,2)--(10,0);
\draw(5,1)--(5,0);
\draw(5,0)--(5.5,0);
\draw(6.5,0)--(7,0);
\end{tikzpicture}
\caption{Graphs with $n_{G_i}(C_3)=2$ and $g_{i,3}=3$. For $F^{(5)}_{m_1,m_2}$ and $F^{(6)}_{m_1,m_2,m_3}$, if $m_1=0$ then $v$ and $u$ are adjacent. For $F^{(6)}_{m_1,m_2,m_3}$, if $m_2=0$ then $u$ and $w$ are adjacent. All other parameters are positive integers.}\label{fig:g23}
\end{figure}
\begin{figure}
\centering
\begin{tikzpicture}
\node at (0,10.5) {$F^{(7)}_m$:};
\filldraw(0,9) circle[radius=2pt]node[left]{$w$};
\filldraw(1,10) circle[radius=2pt];
\filldraw(1,8) circle[radius=2pt];
\filldraw(2,9) circle[radius=2pt]node[above]{$v$};
\filldraw(3,9) circle[radius=2pt]node[above]{$v_1$};
\node at (4,9) {$\cdots$};
\filldraw(5,9) circle[radius=2pt]node[above]{$v_m$};
\filldraw(6,9) circle[radius=2pt]node[above]{$u$};
\filldraw(7,10) circle[radius=2pt];
\filldraw(7,8) circle[radius=2pt];
\draw(0,9)--(1,10);
\draw(0,9)--(1,8);
\draw(1,10)--(2,9);
\draw(1,8)--(2,9);
\draw(1,10)--(1,8);
\draw(2,9)--(3,9);
\draw(3,9)--(3.5,9);
\draw(4.5,9)--(5,9);
\draw(5,9)--(6,9);
\draw(6,9)--(7,10);
\draw(6,9)--(7,8);
\draw(7,10)--(7,8);

\node at (0,6.5) {$F^{(8)}_{m_1,m_2}$:};
\filldraw(0,6) circle[radius=2pt];
\filldraw(0,4) circle[radius=2pt];
\filldraw(1,5) circle[radius=2pt]node[above]{$v$};
\filldraw(2,5) circle[radius=2pt]node[above]{$v_1$};
\node at (3,5) {$\cdots$};
\filldraw(4,5) circle[radius=2pt]node[above]{$v_{m_1}$};
\filldraw(5,5) circle[radius=2pt]node[above]{$w$};
\filldraw(6,6) circle[radius=2pt]node[above]{$u$};
\filldraw(6,4) circle[radius=2pt];
\filldraw(7,6) circle[radius=2pt]node[below]{$u_1$};
\node at (8,6) {$\cdots$};
\filldraw(9,6) circle[radius=2pt]node[below]{$u_{m_2}$};
\filldraw(10,6) circle[radius=2pt]node[above]{$u^\prime$};
\filldraw(11,7) circle[radius=2pt];
\filldraw(11,5) circle[radius=2pt];
\draw(0,6)--(0,4);
\draw(0,6)--(1,5);
\draw(0,4)--(1,5);
\draw(1,5)--(2,5);
\draw(2,5)--(2.5,5);
\draw(3.5,5)--(4,5);
\draw(4,5)--(5,5);
\draw(5,5)--(6,6);
\draw(5,5)--(6,4);
\draw(6,6)--(6,4);
\draw(6,6)--(7,6);
\draw(7,6)--(7.5,6);
\draw(8.5,6)--(9,6);
\draw(9,6)--(10,6);
\draw(10,6)--(11,7);
\draw(10,6)--(11,5);
\draw(11,7)--(11,5);

\node at (0,2.5) {$F^{(9)}_{m_1,m_2,m_3}$:};
\filldraw(0,2) circle[radius=2pt];
\filldraw(0,0) circle[radius=2pt];
\filldraw(1,1) circle[radius=2pt]node[above]{$v$};
\filldraw(2,1) circle[radius=2pt]node[above]{$v_1$};
\node at (3,1) {$\cdots$};
\filldraw(4,1) circle[radius=2pt]node[above]{$v_{m_1}$};
\filldraw(5,1) circle[radius=2pt]node[above]{$w$};
\filldraw(6,1) circle[radius=2pt]node[above]{$u_1$};
\node at (7,1) {$\cdots$};
\filldraw(8,1) circle[radius=2pt]node[above]{$u_{m_2}$};
\filldraw(9,1) circle[radius=2pt]node[above]{$u$};
\filldraw(10,0) circle[radius=2pt];
\filldraw(10,2) circle[radius=2pt];
\filldraw(5,0) circle[radius=2pt]node[right]{$w_1$};
\node at (5,-1) {$\vdots$};
\filldraw(5,-2) circle[radius=2pt]node[right]{$w_{m_3}$};
\filldraw(5,-3) circle[radius=2pt]node[right]{$w^\prime$};
\filldraw(4,-4) circle[radius=2pt];
\filldraw(6,-4) circle[radius=2pt];
\draw(0,2)--(0,0);
\draw(0,2)--(1,1);
\draw(0,0)--(1,1);
\draw(1,1)--(2,1);
\draw(2,1)--(2.5,1);
\draw(3.5,1)--(4,1);
\draw(4,1)--(5,1);
\draw(5,1)--(6,1);
\draw(6,1)--(6.5,1);
\draw(7.5,1)--(8,1);
\draw(8,1)--(9,1);
\draw(9,1)--(10,2);
\draw(9,1)--(10,0);
\draw(10,2)--(10,0);
\draw(5,1)--(5,0);
\draw(5,0)--(5,-0.5);
\draw(5,-1.5)--(5,-2);
\draw(5,-2)--(5,-3);
\draw(5,-3)--(4,-4);
\draw(5,-3)--(6,-4);
\draw(4,-4)--(6,-4);
\end{tikzpicture}
\caption{Graphs with $n_{G_i}(C_3)=3$ and $g_{i,3}=4$. For $F^{(7)}_m$, if $m=0$ then $v$ and $u$ are adjacent. For $F^{(8)}_{m_1,m_2}$ and $F^{(9)}_{m_1,m_2,m_3}$, if $m_1=0$ then $v$ and $w$ are adjacent. For $F^{(8)}_{m_1,m_2}$, if $m_2=0$ then $u$ and $u^\prime$ are adjacent. For $F^{(9)}_{m_1,m_2,m_3}$, if $m_2=0$ then $u$ and $w$ are adjacent, and if $m_3=0$ then $w$ and $w^\prime$ are adjacent.}\label{fig:g34}
\end{figure}
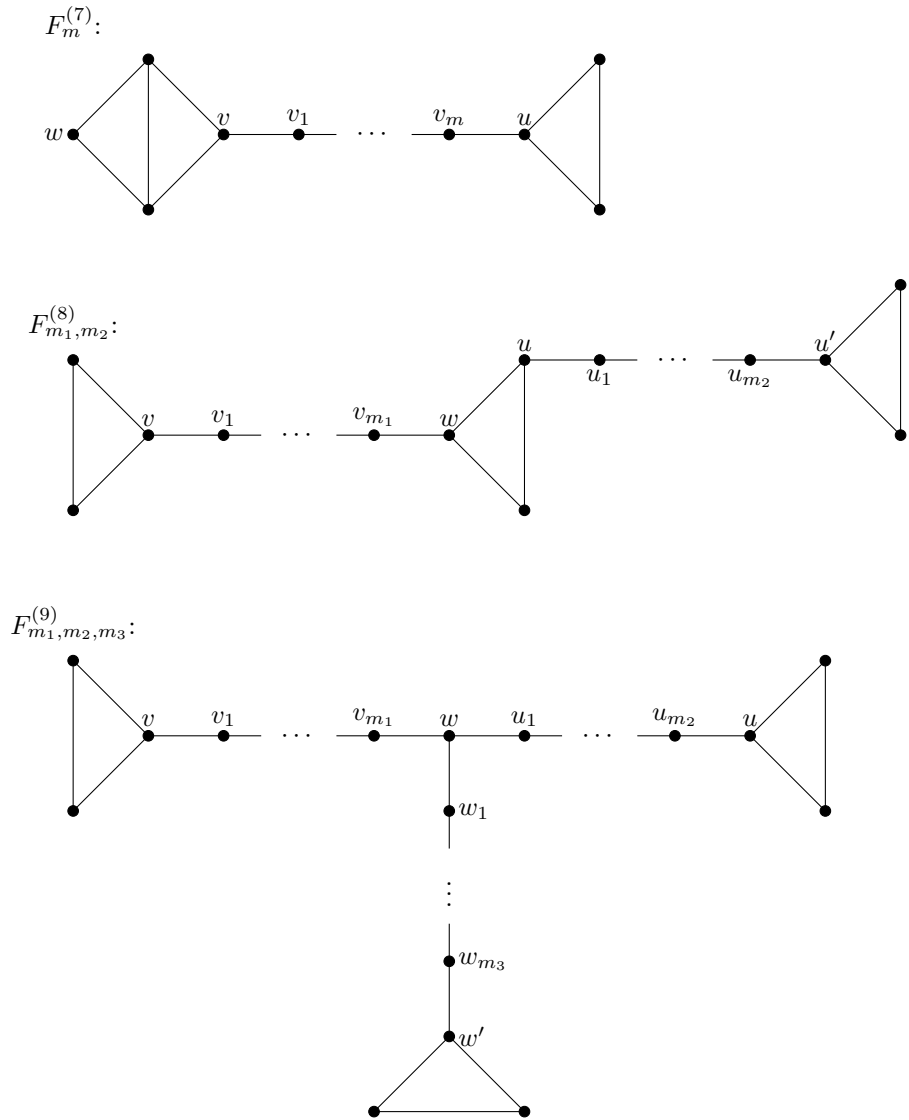

\section{Candidate connected components of graphs $G$ in $\mathcal{I}(P_n)$}

From Proposition~\ref{prop:pathroot}, we find that the roots of $I(P_n,x)$ are all real and smaller than $-\frac{1}{4}$. Therefore, if we can show that any of the graphs have independence roots which are not real, or real numbers greater than $-\frac{1}{4}$, then they are eliminated as possible connected components of $G$. In particular, since for any graph $G$, the value of $I(G,x)$ at $x=0$ is $1$, if we can show that the value of $I(G_i,x)$ at $x=-\frac{1}{4}$ is negative, then $I(G_i,x)$ has a root in the interval $\left[-\frac{1}{4},0\right)$. 

We have one tool that we will use repeatedly in this section. Suppose we have a sequence of graphs $F_0, F_1, \ldots$ such that
\begin{equation}\label{eq:recur}
I(F_m,x)=I(F_{m-1},x)+xI(F_{m-2},x)    
\end{equation}
for $m\geq 2$. If we let $\tau_m$ be the numerical value of $I\left(F_m,-\frac{1}{4}\right)$, then $\tau_m$ obeys the recurrence relation
$$\tau_m = \tau_{m-1} - \frac{1}{4}\tau_{m-2}.$$
The solution to this recurrence relation is
\begin{equation}\label{eq:recursol}
\tau_m = \frac{Um+V}{2^m} 
\end{equation}
where $U$ and $V$ are determined by the initial values of $\tau_m$.
\begin{prop}\label{prop:recur1}
If $\tau_m = \tau_{m+1}$ for some value of $m$, then all subsequent terms will have the same sign as $\tau_m$. 
\end{prop}
\begin{proof}
Let $\tau_m = \tau_{m+1} = \tau$. Then
$$\tau = \frac{Um+V}{2^m} = \frac{U(m+1)+V}{2^{m+1}}$$
so $V=U(1-m)$ and, therefore,
$$\tau = \frac{U}{2^m},$$
which is of the same sign as $U$. For $k\geq 2$,
$$\tau_{m+k} = \frac{U(m+k)+V}{2^{m+k}} = \frac{U(k+1)}{2^{m+k}},$$
which is of the same sign as $U$, and hence, as $\tau$.
\end{proof}
\begin{prop}\label{prop:recur2}
For the path $P_n$, the cycle graph $C_n$, and the graph $D_n$,
\begin{eqnarray*}
I\left(P_n,-\frac{1}{4}\right) & = & \frac{n+2}{2^n}, \\
I\left(C_n,-\frac{1}{4}\right) = I\left(D_n,-\frac{1}{4}\right) & = & \frac{1}{2^{n-1}}.
\end{eqnarray*}
\end{prop}
\begin{proof}
The graphs $P_n$, $C_n$ and $D_n$ each satisfy the recurrence relation~(\ref{eq:recur}). Therefore, the value of their independence polynomial at $x=-\frac{1}{4}$ follows the relation in~(\ref{eq:recursol}). Working out the first two values to find $U$ and $V$ in each case gives us the result.
\end{proof}
For convenience, we define $D_3=C_3$, $D_2=P_2$ and $P_0$ to be the empty graph so that $I(P_0,x)=1$.

\subsection{Eliminating the $F^{(j)}$ graphs}

It turns out that all the graphs labelled as $F^{(j)}$ (where $1\leq j \leq 9$) with various subscripts, can be eliminated because of the location of the roots of their independence polynomials.

$F^{(1)}_{m_1,m_2}$ is independence equivalent to $F^{(5)}_{m_1,m_2}$, since deleting $u$ from both graphs gives us independence equivalent graphs, and deleting $N[u]$ from both graphs gives us isomorphic graphs. Therefore, we will discuss $F^{(1)}_{m_1,m_2}$ when we discuss $F^{(5)}_{m_1,m_2}$ later.

$F^{(2)}_m$ is independence equivalent to $F^{(4)}_m$, since deleting $v$ from both graphs gives us independence equivalent graphs, and deleting $N[v]$ from both graphs gives us isomorphic graphs. Therefore, we will discuss $F^{(2)}_m$ when we discuss $F^{(4)}_m$ later.

In the case of $F^{(3)}_m$, we can apply Proposition~\ref{prop:delvert} on $v$ to find that
$$I\left(F^{(3)}_m,x\right)  =  I\left(D_{m+5},x\right)+xI\left(D_{m+3},x\right).$$
Therefore,
\begin{eqnarray*}
I\left(F^{(3)}_m,-\frac{1}{4}\right) & = &  I\left(D_{m+5},-\frac{1}{4}\right)+xI\left(D_{m+3},-\frac{1}{4}\right)\\
& = & \frac{1}{2^{m+4}} - \frac{1}{4}\frac{1}{2^{m+2}}\\
& = & 0.
\end{eqnarray*}
Therefore, $-\frac{1}{4}$ is a root of $I\left(F^{(3)}_m,x\right)$ for all non-negative integers $m$. Thus, it cannot be a factor of $I(P_n,x)$ for any $n$.

In the case of $F^{(4)}_m$, we find that, for $m\geq 3$, we can apply Proposition~\ref{prop:delvert} on the vertex of degree~1 to find that $I\left(F^{(4)}_m,x\right)$ satisfies the recurrence relation~(\ref{eq:recur}). We also find that $\tau_1=0$ and $\tau_2=\tau_3=-\frac{1}{64}<0$. By Proposition~\ref{prop:recur1}, $\tau_m\leq 0$ for all $m\geq 1$. Therefore, $I\left(F^{(4)}_m,x\right)$ always has a root in the interval $\left[-\frac{1}{4},0\right)$ and cannot be a factor of $I\left(P_n,x\right)$ for any $n$.

In the case of $F^{(5)}_{m_1,m_2}$, applying Proposition~\ref{prop:delvert} on $u_{m_2}$, we have
$$I\left(F^{(5)}_{m_1,1},x\right) = I\left(F^{(3)}_{m_1},x\right) + xI\left(D_{m_1+5},x\right)$$
for $m_2=1$,
$$I\left(F^{(5)}_{m_1,2},x\right) = I\left(F^{(5)}_{m_1,1},x\right) + xI\left(F^{(3)}_{m_1},x\right)$$
for $m_2=2$,
$$I\left(F^{(5)}_{m_1,m_2},x\right) = I\left(F^{(5)}_{m_1,m_2-1},x\right) + xI\left(F^{(5)}_{m_1,m_2-2},x\right)$$
for $m_2\geq 3$, so that $I\left(F^{(5)}_{m_1,m_2},x\right)$ satisfies the recurrence relation~(\ref{eq:recur}) where $m_2$ is the index. Substituting $x=-\frac{1}{4}$, we find that $\tau_1=\tau_2=-2^{-(m_1+6)} < 0$. By Proposition~\ref{prop:recur1}, $\tau_{m_2}$ is negative for all $m_2\geq 1$. Therefore, $I\left(F^{(5)}_{m_1,m_2},x\right)$ always has a root in the interval $\left[-\frac{1}{4},0\right)$ and cannot be a factor of $I\left(P_n,x\right)$ for any $n$.

In the case of $F^{(6)}_{m_1,m_2,m_3}$, applying Proposition~\ref{prop:delvert} on $w_{m_3}$, we have
$$I\left(F^{(6)}_{m_1,m_2,1},x\right) = I\left(F^{(3)}_{m_1+m_2+1},x\right) + xI\left(D_{m_1+3},x)I(D_{m_2+3},x\right)$$
for $m_3=1$,
$$I\left(F^{(6)}_{m_1,m_2,2},x\right) = I\left(F^{(6)}_{m_1,m_2,1},x\right) + xI\left(F^{(3)}_{m_1+m_2+1},x\right)$$
for $m_3=2$,
$$I\left(F^{(6)}_{m_1,m_2,m_3},x\right) = I\left(F^{(6)}_{m_1,m_2,m_3-1},x\right) + xI\left(F^{(6)}_{m_1,m_2,m_3-2},x\right)$$
for $m_3\geq 3$, so that $I\left(F^{(6)}_{m_1,m_2,m_3},x\right)$ satisfies the recurrence relation~(\ref{eq:recur}) where $m_3$ is the index. Substituting $x=-\frac{1}{4}$, we find that $\tau_1=\tau_2=-2^{-(m_1+m_2+6)}<0$. By Proposition~\ref{prop:recur1}, $\tau_{m_3}$ is negative for all $m_3\geq 1$. Therefore, $I\left(F^{(6)}_{m_1,m_2,m_3},x\right)$ always has a root in the interval $\left[-\frac{1}{4},0\right)$ and cannot be a factor of $I\left(P_n,x\right)$ for any $n$.

In the case of $F^{(7)}_m$, applying Proposition~\ref{prop:delvert} on $w$, we find that
$$I\left(F^{(7)}_{m},x\right) = I\left(F^{(3)}_{m},x\right) + xI\left(D_{m+3},x\right).$$
Substituting $x=-\frac{1}{4}$, we have
$$I\left(F^{(7)}_{m},-\frac{1}{4}\right) = I\left(F^{(3)}_{m},-\frac{1}{4}\right) + -\frac{1}{4}I\left(D_{m+3},-\frac{1}{4}\right) = -\frac{1}{2^{m+4}}<0.$$
Therefore, $I\left(F^{(7)}_{m},x\right)$ has a root in the interval $\left[-\frac{1}{4},0\right)$ and cannot be a factor of $I\left(P_n,x\right)$ for any $n$.

In the case of $F^{(8)}_{m_1,m_2}$, if $m_1=m_2=0$, then we can explicitly calculate
$$I\left(F^{(8)}_{0,0},x\right)=1+9x+25x^2+21x^3=(1+3x)(1+6x+7x^2)$$
which turns out to have a root $\frac{1}{7}\left(\sqrt{2}-3\right)\in\left[-\frac{1}{4},1\right).$

Otherwise, without loss of generality, suppose that $m_2\geq 1$. Applying Proposition~\ref{prop:delvert} on $u_1$, we have
$$I\left(F^{(8)}_{m_1,m_2},x\right) = I\left(F^{(3)}_{m_1},x\right)I\left(D_{m_2+2},x\right) + xI\left(D_{m_1+5},x\right)I\left(D_{m_2+1},x\right).$$
Substituting $x=-\frac{1}{4}$, we have
\begin{eqnarray*}
I\left(F^{(8)}_{m_1,m_2},-\frac{1}{4}\right) & = & I\left(F^{(3)}_{m_1},-\frac{1}{4}\right)I\left(D_{m_2+2},-\frac{1}{4}\right)\\
& &\qquad-\frac{1}{4}I\left(D_{m_1+5},-\frac{1}{4}\right)I\left(D_{m_2+1},-\frac{1}{4}\right)\\
& = & -\frac{1}{2^{m_1+m_2+6}}\\
& < & 0.
\end{eqnarray*}
Therefore, $I\left(F^{(8)}_{m_1,m_2},x\right)$ has a root in the interval $\left[-\frac{1}{4},0\right)$. In both cases, $I\left(F^{(8)}_{m_1,m_2},x\right)$ cannot be a factor of $I\left(P_n,x\right)$ for any $n$.

In the case of $F^{(9)}_{m_1,m_2,m_3}$, if $m_1=m_2=m_3=0$, then we can explicitly calculate
$$I\left(F^{(9)}_{0,0,0},x\right)=1+10x+33x^2+39x^3+8x^4$$
which has two non-real roots and thus cannot be a factor of $I\left(P_n,x\right)$.

Otherwise, without loss of generality, suppose that $m_3\geq 1$. Applying Proposition~\ref{prop:delvert} on $w_1$, we have
\begin{eqnarray*}
I\left(F^{(9)}_{m_1,m_2,m_3},x\right) & = & I\left(F^{(3)}_{m_1+m_2+1},x\right)I\left(D_{m_3+2},x\right) \\
& & \qquad +xI\left(D_{m_1+3},x\right)I\left(D_{m_2+3},x\right)I\left(D_{m_3+1},x\right).
\end{eqnarray*}
Substituting $x=-\frac{1}{4}$, we have
\begin{eqnarray*}
I\left(F^{(9)}_{m_1,m_2,m_3},-\frac{1}{4}\right) & = & I\left(F^{(3)}_{m_1+m_2+1},-\frac{1}{4}\right)I\left(D_{m_3+2},-\frac{1}{4}\right) \\
& & -\frac{1}{4}I\left(D_{m_1+3},-\frac{1}{4}\right)I\left(D_{m_2+3},-\frac{1}{4}\right)I\left(D_{m_3+1},-\frac{1}{4}\right)\\
& = & -\frac{1}{2^{m_1+m_2+m_3+6}}\\
& < & 0.
\end{eqnarray*}
Therefore, $I\left(F^{(9)}_{m_1,m_2,m_3},x\right)$ has a root in the interval $\left[-\frac{1}{4},0\right)$. In both cases, $I\left(F^{(9)}_{m_1,m_2,m_3},x\right)$ cannot be a factor of $I\left(P_n,x\right)$ for any $n$.

\subsection{The graph $Y_{m_1,m_2,m_3}$}

In this subsection, we will try to find values of $m_1$, $m_2$ and $m_3$ such that $Y_{m_1,m_2,m_3}$ could potentially be a connected component of a graph $G\in\mathcal{I}(P_n)$.

By applying Proposition~\ref{prop:delvert} on $v$, we find that
\begin{eqnarray*}
I(Y_{m_1,m_2,m_3},x) & = & I(P_{m_1},x)I(P_{m_2},x)I(P_{m_3},x) \\
& &\qquad +x(P_{m_1-1},x)I(P_{m_2-1},x)I(P_{m_3-1},x).
\end{eqnarray*}
Substituting $x=-\frac{1}{4}$, and applying Proposition~\ref{prop:recur2},
\begin{eqnarray*}
& & I\left(Y_{m_1,m_2,m_3},-\frac{1}{4}\right) \\ 
& = & \left(\frac{m_1+2}{2^{m_1+1}}\right)\left(\frac{m_2+2}{2^{m_2+1}}\right)\left(\frac{m_3+2}{2^{m_3+1}}\right) -\frac{1}{4}\left(\frac{m_1+1}{2^{m_1}}\right)\left(\frac{m_2+1}{2^{m_2}}\right)\left(\frac{m_3+1}{2^{m_3}}\right)\\
& = & \frac{(m_1+2)(m_2+2)(m_3+2)-2(m_1+1)(m_2+1)(m_3+1)}{2^{m_1+m_2+m_3+3}}
\end{eqnarray*}
Since the denominator is positive, we only need to be concerned with the sign of the numerator. If the numerator is negative, then $I(Y_{m_1,m_2,m_3},x)$ has a root in the interval $\left[-\frac{1}{4},0\right)$ and hence cannot be a factor of $I\left(P_n,x\right)$ for any $n$. Therefore, we wish to find values of $m_1, m_2, m_3$ for which the numerator is positive. Once we find those, we find the roots of $I(Y_{m_1,m_2,m_3},x)$ to ensure that all of them lie in the interval $(-\infty,-\frac{1}{4})$, and only then can $Y_{m_1,m_2,m_3}$ be considered a candidate connected component of $G$. Now,
$$(m_1+2)(m_2+2)(m_3+2)-2(m_1+1)(m_2+1)(m_3+1) > 0$$
is true if and only if
$$m_1m_2m_3 < 2(m_1 + m_2 + m_3)+6$$
or
$$m_1(m_2m_3-2) < 2(m_2 + m_3 + 3).$$
Without loss of generality, suppose that $m_1\geq m_2\geq m_3$.

If $m_3\geq 2$ then $m_2m_3-2>0$ so we can rearrange the inequality to give us
\begin{equation}\label{eq:m1}
m_1 < \frac{2(m_2+m_3+3)}{m_2m_3-2} = \frac{2}{m_3}+\frac{2m_3+\frac{4}{m_3}+6}{m_2m_3-2}. 
\end{equation}
If we treat $m_3$ as a constant, the right side of the inequality decreases as $m_2$ increases. Hence, its maximum value occurs when at the smallest value of $m_2$, that is, $m_2=m_3$. Also, since $m_3 \leq m_1$, we see that $m_3$ must satisfy
$$m_3 < \frac{2}{m_3}+\frac{2m_3+\frac{4}{m_3}+6}{m_3^2-2}.$$
For positive $m_3$, the solutions are in the interval $m_3\in\left(0, \sqrt[3]{2}+\sqrt[3]{4}\right]$. Since we are working under the assumption that $m_3\geq 2$, the only possible value for $m_3$ here is $2$. In that case, our inequality~(\ref{eq:m1}) becomes
$$m_1 < \frac{2m_2 + 10}{2m_2-2}=1+\frac{12}{2m_2-2}.$$
When $m_2=2$, we have $m_1<7$ so the possible values of $m_1$ are $2,3,4,5,6$. When $m_2=3$, we have $m_1<4$ so the only possible value of $m_1$ is $3$. For $m_2 \geq 4$ we find that $m_1 < 4$, a contradiction. By directly calculating $I\left(Y_{m_1,m_2,m_3},x\right)$ in these cases, we find that only $Y_{4,2,2}$, $Y_{3,3,2}$ and $Y_{3,2,2}$ have all roots in the interval $(-\infty,-\frac{1}{4})$.

If $m_3=1$ and $m_2\geq 3$, we still have $m_2m_3-2>0$ so the inequality~(\ref{eq:m1}) still holds and becomes
$$m_1 < \frac{2m_2+8}{m_2-2} = 2+\frac{12}{m_2-2}.$$
When $m_2 = 3$, we have $m_1<14$. When $m_2 = 4$, we have $m_1 < 8$. When $m_2 = 5$, we have $m_1 < 6$ For $m_2 \geq 6$ we find that $m_1 < 5$, a contradiction.  By directly calculating $I\left(Y_{m_1,m_2,m_3},x\right)$ in these cases, we find that only $Y_{5,4,1}$, $Y_{9,3,1}$, $Y_{7,3,1}$ and $Y_{4,3,1}$ have all roots in the interval $(-\infty,-\frac{1}{4})$.

If $m_3=1$ and $m_2=2$, it turns out that $Y_{m_1,2,1}$ is independence equivalent to $P_1\cup C_{m_1+3}$. This can be proven by applying Proposition~\ref{prop:deledge} on the edge between the vertex of degree~3 and the leaf in $Y_{m_1,2,1}$, and any edge in $P_1\cup C_{m_1+3}$. In both graphs, deleting the edge gives $P_1\cup P_{n+3}$, while deleting the neighbourhoods of both vertices incident to the edge gives $P_1\cup P_{n-1}$. Therefore, since $P_1$ and cycle graphs can clearly be connected components of $G$, all graphs of the form $Y_{m_1,2,1}$ should be considered.

The final case, $m_3=1$ and $m_2=1$, is the most complicated. Let $p_m(x)=I(Y_{m,1,1},x)$ for $m\geq 1$. 
The recurrence relation~(\ref{eq:recur}) applies to $p_m(x)$, so we have
\begin{eqnarray*}
p_1(x) & = & 1 + 4x + 3x^2 + x^3, \\
p_2(x) & = & 1 + 5x + 6x^2 + 2x^3, \\
p_m(x) & = & p_{m-1}(x) + xp_{m-2}(x)
\end{eqnarray*}
for $m\geq 2$. (This can be seen by applying Proposition~\ref{prop:delvert} on the leaf not adjacent to the vertex of degree~3.)

The solution to the recurrence relation for $p_m(x)$ is
$$p_m(x)=A(x)\left(\frac{1+\sqrt{1+4x}}{2}\right)^m+B(x)\left(\frac{1-\sqrt{1+4x}}{2}\right)^m$$
where
\begin{eqnarray*}
A(x) & = & \frac{1}{2}\left(1 + 3x + x^2 + \frac{1 + 5x + 5x^2 + 2x^3}{\sqrt{1+4x}}\right),\\
B(x) & = & \frac{1}{2}\left(1 + 3x + x^2 - \frac{1 + 5x + 5x^2 + 2x^3}{\sqrt{1+4x}}\right).
\end{eqnarray*}

We wish to know how many real roots $p_m(x)$ has that are real and less than $-\frac{1}{4}$. The equation $p_m(x)=0$ can be rearranged into
$$\left(\frac{1+\sqrt{1+4x}}{1-\sqrt{1+4x}}\right)^m=-\frac{B(x)}{A(x)}.$$
Now, if $x<-\frac{1}{4}$, then $\sqrt{1+4x}$ will be purely imaginary and we can let $\sqrt{1+4x}=ki$ for some positive real number $k$. Then we can substitute $x=-\frac{1}{4}(1+k^2)$ and solve for (positive) $k$ instead. Our equation now becomes
$$\left(\frac{1+ki}{1-ki}\right)^m=\frac{1+5x+5x^2+2x^3-(1+3x+x^2)ki}{1+5x+5x^2+2x^3+(1+3x+x^2)ki}.$$
Since both $x$ and $k$ are real numbers, the numerator and denominator on both sides are complex conjugates. Therefore, the modulus on both sides is $1$ and we only need to equate the arguments. Doing so and substituting $x=-\frac{1}{4}(1+k^2)$ gives

\begin{eqnarray*}
m \tan^{-1}(k) & = & \tan^{-1}\left(-\frac{2k(5 - 10k^2 + k^4)}{1 - 23k^2 + 7k^4 - k^6}\right)\\
\tan\left(m \tan^{-1}(k)\right) & = & -\frac{2k(5 - 10k^2 + k^4)}{1 - 23k^2 + 7k^4 - k^6}.
\end{eqnarray*}

Now, the function
$$F(k) = -\frac{2k(5 - 10k^2 + k^4)}{1 - 23k^2 + 7k^4 - k^6}$$
for positive $k$ is not defined when $1 - 23k^2 + 7k^4 - k^6=0$, which has one positive root $\alpha$ (which is approximately $0.21$). For $k\in(0,\alpha)$, $F(k)$ is strictly decreasing from $0$ to $-\infty$. For $k>\alpha$, the behaviour of $F(k)$ is more complicated. It has a local minimum at $\beta\approx1.888$ and a local maximum at $\gamma\approx4.637$. For $k\in(\alpha,\beta]$, $F(k)$ is strictly decreasing from $+\infty$ to approximately $-1.814$. For $k\in[\beta,\gamma]$, $F(k)$ is strictly increasing from approximately $-1.814$ to approximately $0.325$, and for $k\geq\gamma$, it is strictly decreasing from $0.325$ to approach $0$ asymptotically. The only region where $F(k)$ is increasing is $[\beta,\gamma]$, and in this region there is only one point of inflection, at approximately $k=2.358$. This is where $F^\prime(k)$ reaches its maximum value of approximately $2.27$.

The function
$$G_m(k) = \tan\left(m\tan^{-1}(k)\right)$$
for positive $k$ is not defined when
\begin{equation}\label{eq:kvalue}
k=\tan\left(\frac{\pi}{2m}\right),\tan\left(\frac{3\pi}{2m}\right),\tan\left(\frac{5\pi}{2m}\right),\ldots,\tan\left(\left(2\left\lfloor\frac{m}{2}\right\rfloor-1\right)\frac{\pi}{2m}\right).
\end{equation}
Between each pair of consecutive values listed in~(\ref{eq:kvalue}), $G_m(k)$ is strictly increasing from $-\infty$ to $+\infty$. For $k\in\left(0,\tan\left(\frac{\pi}{2m}\right)\right)$, $G_m(k)$ is strictly increasing from $0$ to $+\infty$, and for $k>\tan\left(\left(2\left\lfloor\frac{m}{2}\right\rfloor-1\right)\frac{\pi}{2m}\right)$, $G_m(k)$ is strictly increasing from $-\infty$ to $+\infty$ for odd $m$, and from $-\infty$ to $0$ for even $m$.

\begin{figure}
    \centering
    \includegraphics[width=\textwidth]{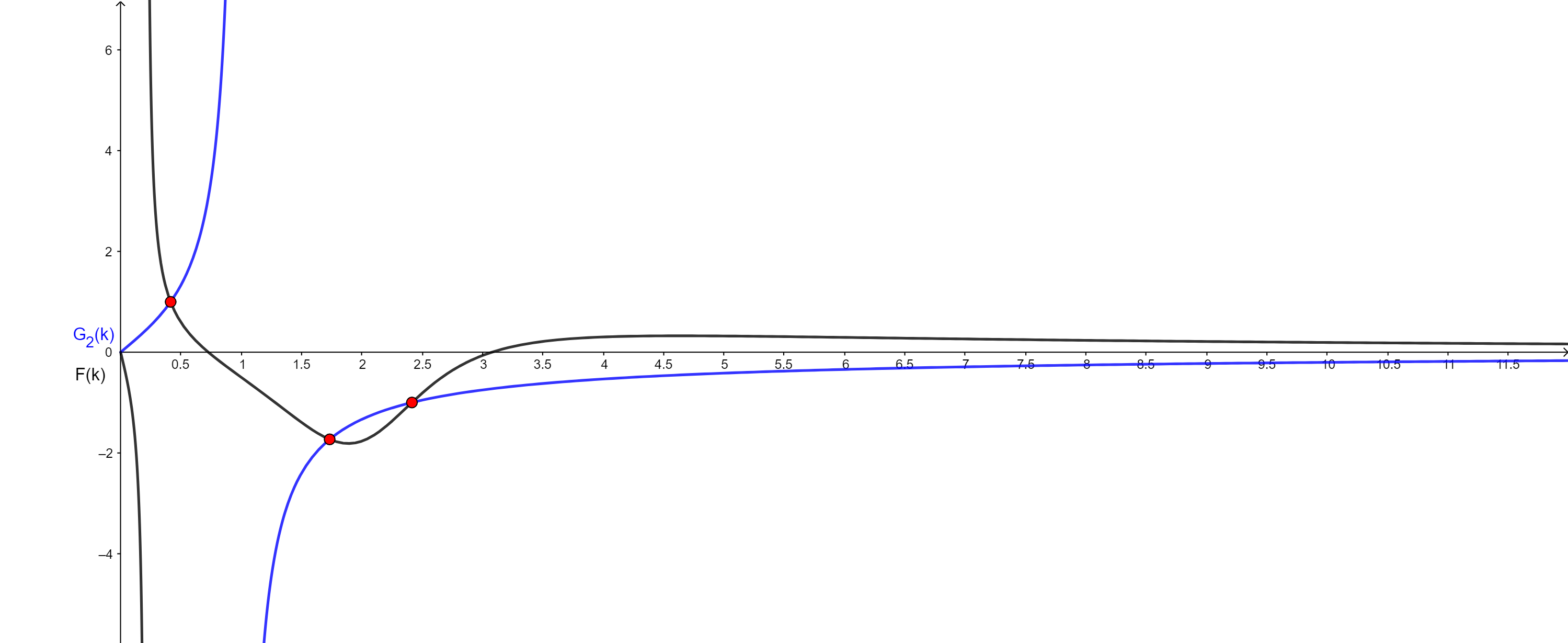}
    \caption{Graph of $F(k)$ (black) and $G_{2}(k)$ (blue), with intersections shown in red.}
    \label{fig:g2graph}
\end{figure}

\begin{figure}
    \centering
    \includegraphics[width=\textwidth]{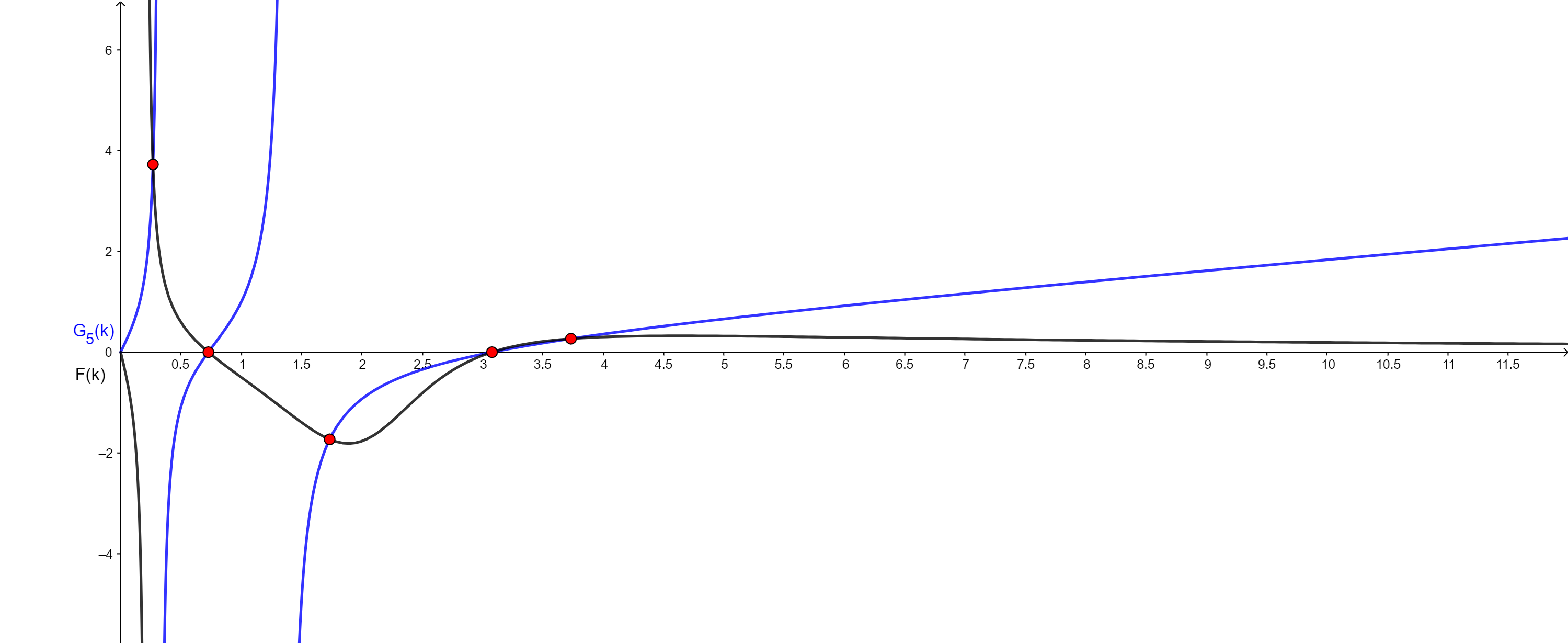}
    \caption{Graph of $F(k)$ (black) and $G_{5}(k)$ (blue), with intersections shown in red.}
    \label{fig:g5graph}
\end{figure}

\begin{figure}
    \centering
    \includegraphics[width=\textwidth]{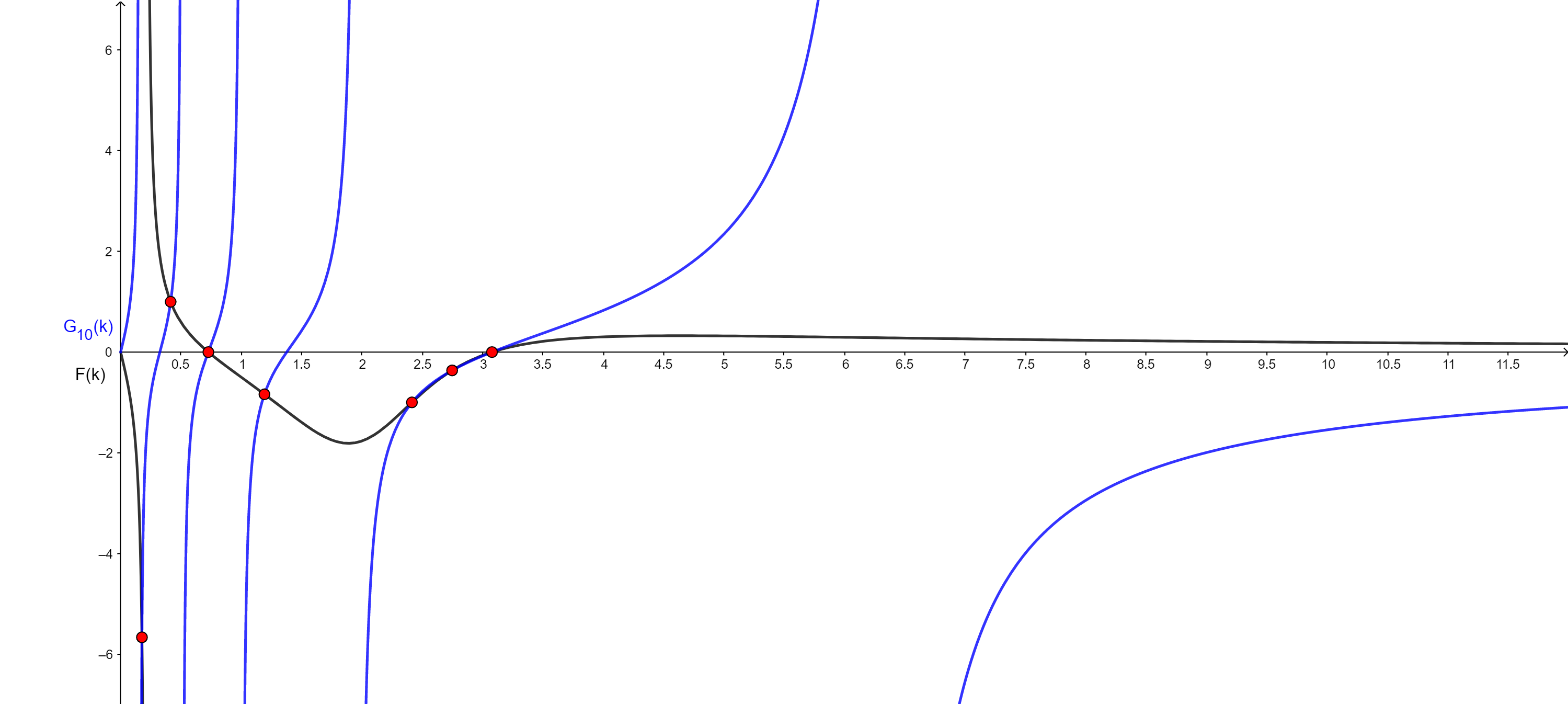}
    \caption{Graph of $F(k)$ (black) and $G_{10}(k)$ (blue), with intersections shown in red.}
    \label{fig:g10graph}
\end{figure}

\begin{figure}
    \centering
    \includegraphics[width=\textwidth]{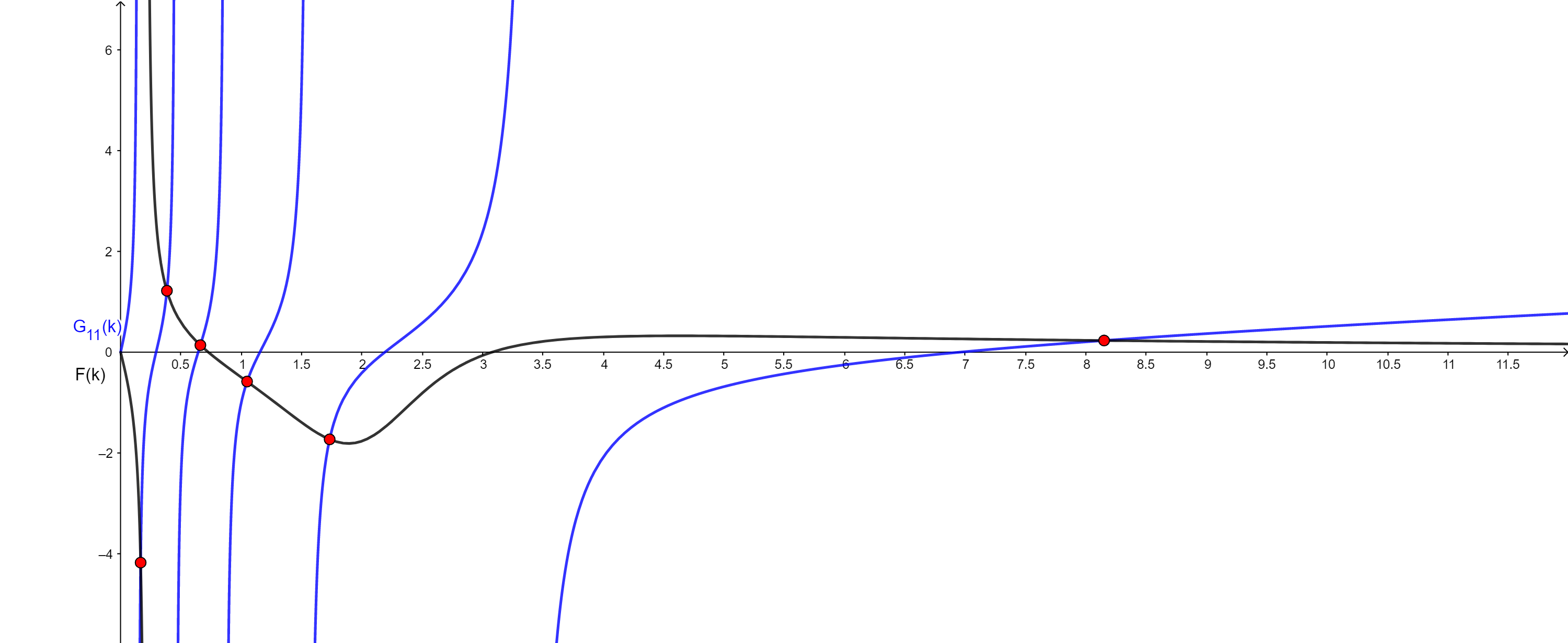}
    \caption{Graph of $F(k)$ (black) and $G_{11}(k)$ (blue), with intersections shown in red.}
    \label{fig:g11graph}
\end{figure}

\begin{figure}
    \centering
    \includegraphics[width=\textwidth]{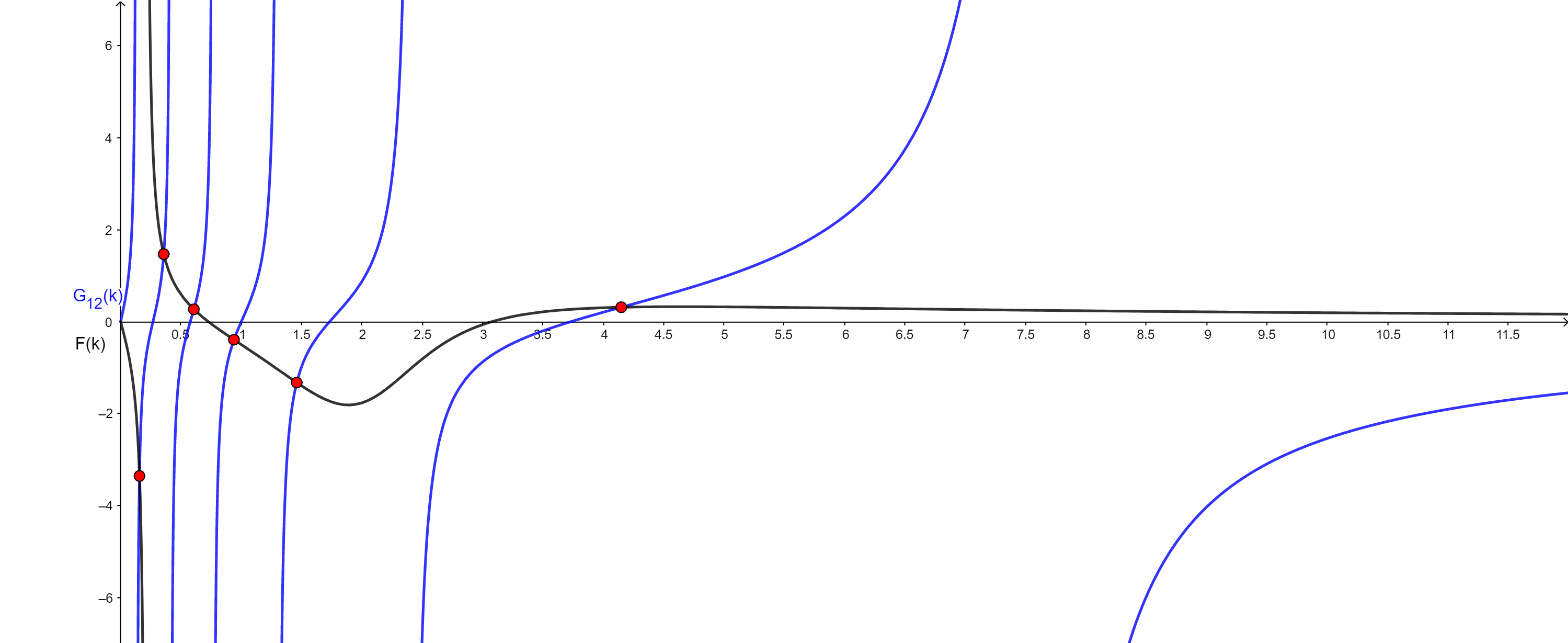}
    \caption{Graph of $F(k)$ (black) and $G_{12}(k)$ (blue), with intersections shown in red.}
    \label{fig:g12graph}
\end{figure}

\begin{figure}
    \centering
    \includegraphics[width=\textwidth]{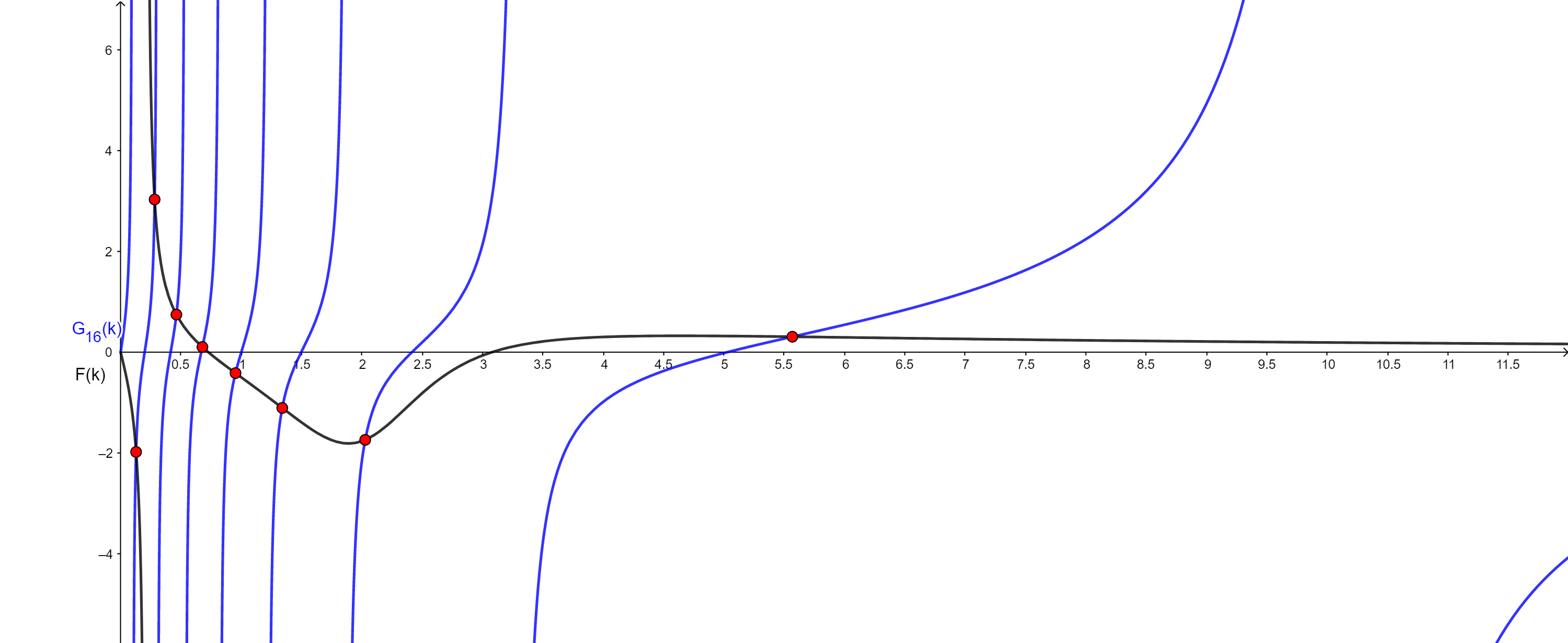}
    \caption{Graph of $F(k)$ (black) and $G_{16}(k)$ (blue), with intersections shown in red.}
    \label{fig:g16graph}
\end{figure}

\begin{figure}
    \centering
    \includegraphics[width=\textwidth]{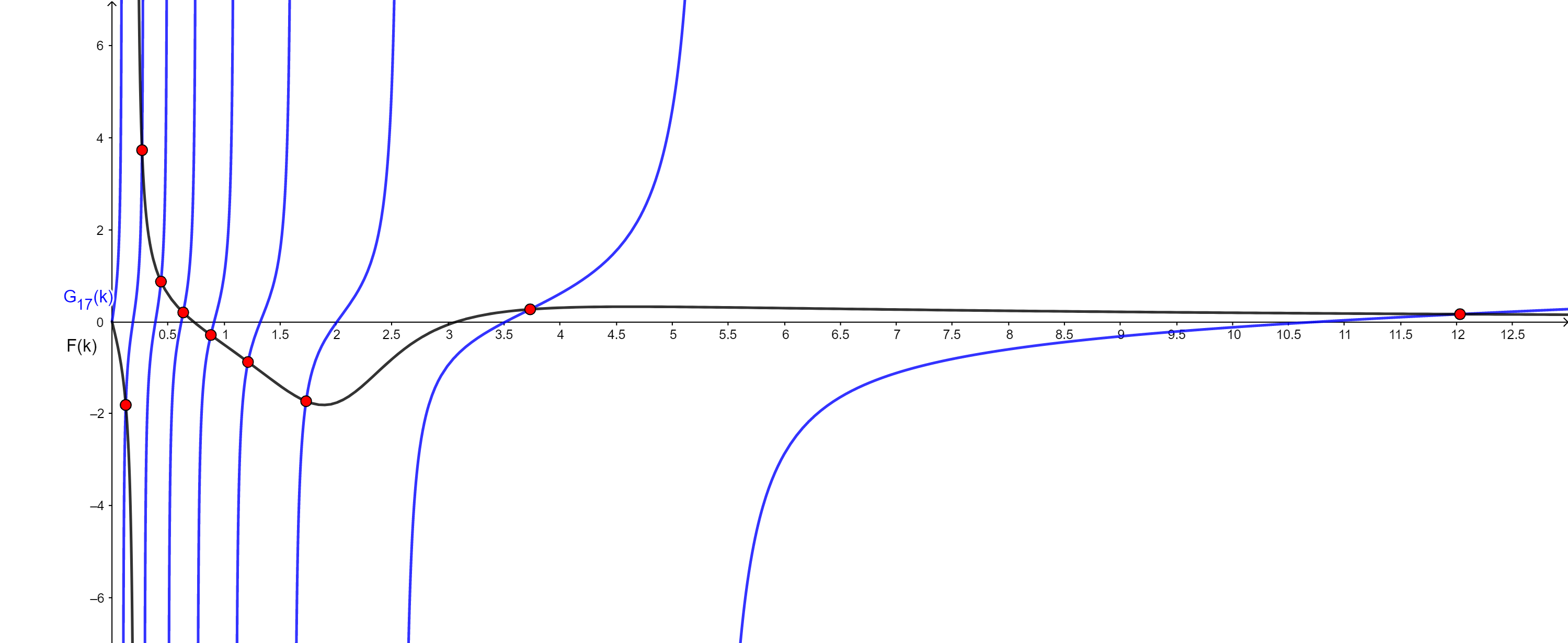}
    \caption{Graph of $F(k)$ (black) and $G_{17}(k)$ (blue), with intersections shown in red.}
    \label{fig:g17graph}
\end{figure}

We would like to determine how many times the functions $F(k)$ and $G_m(k)$ take on the same value. Now for $m\geq 16$, it turns out that
$$G_m^\prime(k) = \frac{m}{1+k^2}\sec^2(m\tan^{-1}(k)) \geq \frac{m}{1+k^2} > F^\prime(k)$$
for all values of $k$ where both functions are defined.

I claim that for $m\geq 16$, the functions $F(k)$ and $G_m(k)$ take the same value $\frac{m+1}{2}$ times if $m$ is odd and $\frac{m}{2}$ times if $m$ is even. Note that $m$ is large enough that $\alpha\approx0.21$ falls in between some pair of values $\tan\left((2j-1)\frac{\pi}{2m}\right)$ and $\tan\left((2j+1)\frac{\pi}{2m}\right)$ for some integer $j$ where $1\leq j\leq \left\lfloor\frac{m}{2}\right\rfloor-1$. Now for $k<\tan\left(\frac{\pi}{2m}\right)$, $F(k)<0$ and $G_m(k)>0$ so they cannot take the same value. Between each pair of consecutive values listed in~(\ref{eq:kvalue}) from $\tan\left(\frac{\pi}{2m}\right)$ to $\tan\left((2j-1)\frac{\pi}{2m}\right)$, and from $\tan\left((2j+1)\frac{\pi}{2m}\right)$ to $\tan\left(\left(2\left\lfloor\frac{m}{2}\right\rfloor-1\right)\frac{\pi}{2m}\right)$, $F(k)$ and $G_m(k)$ take the same value exactly once. This is because both functions are continuous between each pair of values, $G_m^\prime(k)>F^\prime(k)$, and $G_m(k)$ is strictly increasing from $-\infty$ to $+\infty$. Between $\tan\left((2j-1)\frac{\pi}{2m}\right)$ and $\alpha$, $F(k)$ and $G_m(k)$ take the same value exactly once. This is because both functions are continuous between each pair of values, $F(k)$ is strictly decreasing from some finite value to $-\infty$, and $G_m(k)$ is strictly increasing from $-\infty$ to some finite value. Similarly, between $\alpha$ and $\tan\left((2j+1)\frac{\pi}{2m}\right)$, $F(k)$ and $G_m(k)$ take the same value exactly once. This is because both functions are continuous between each pair of values, $F(k)$ is strictly decreasing from $+\infty$ to some finite value, and $G_m(k)$ is strictly increasing from some finite value to $+\infty$. Finally, for $k>\tan\left(\left(2\left\lfloor\frac{m}{2}\right\rfloor-1\right)\frac{\pi}{2m}\right)$, if $m$ is odd, then  $F(k)$ and $G_m(k)$ take the same value exactly once. This is because $m$ is large enough that $\tan\left(\left(2\left\lfloor\frac{m}{2}\right\rfloor-1\right)\frac{\pi}{2m}\right)>\gamma$, so $G_m(k)$ is strictly increasing from $-\infty$ to $+\infty$ while $F(k)$ is strictly decreasing from some finite positive value to $0$. If $m$ is even, then $F(k)$ and $G_m(k)$ do not take the same value in that region as $F(k) > 0$ whereas $G_m(k) < 0$. Counting the total number of times $F(k)$ and $G_m(k)$ take the same value proves the claim. Refer to Figure~\ref{fig:g16graph} for $m=16$ and Figure~\ref{fig:g17graph} for $m=17$ as representatives of the even and odd cases respectively.

Therefore, for $m\geq 16$, $p_m(x)=0$ has $\frac{m+1}{2}$ real roots less than $-\frac{1}{4}$ if $m$ is odd and $\frac{m}{2}$ real roots less than $-\frac{1}{4}$ if $m$ is even. However, we can see from the recurrence relation defining $p_m(x)$ that the degree of $p_m(x)$ is $\frac{m+1}{2}+2$ if $m$ is odd and $\frac{m}{2}+2$ if $m$ is even. Therefore, there have to be two more roots which are not real and less than $-\frac{1}{4}$. This eliminates $Y_{m,1,1}$ as a possible connected component of $G$ for $m\geq 16$. For $1\leq m\leq 15$, we find that the independence polynomial has all real roots less than $-\frac{1}{4}$ for $m=2,5,10$ and no other values of $m$. An examination of Figures~\ref{fig:g2graph}, \ref{fig:g5graph} and~\ref{fig:g10graph} would reveal what makes those values of $m$ special, in comparison to more `typical' values of $m$ such as $11$ (Figure~\ref{fig:g11graph}) and $12$ (Figure~\ref{fig:g12graph}).

\subsection{The graph $B_{m_1,m_2,m_3}$}

In this subsection, we will try to find values of $m_1$, $m_2$ and $m_3$ such that $B_{m_1,m_2,m_3}$ could potentially be a connected component of a graph $G\in\mathcal{I}(P_n)$.

By applying Proposition~\ref{prop:delvert} on $u$, we find that
\begin{eqnarray*}
I\left(B_{m_1,m_2,m_3},x\right) & = & I\left(D_{m_1+3},x\right)I\left(P_{m_2},x\right)I\left(P_{m_3},x\right)\\
& & + xI\left(D_{m_1+2},x\right)I\left(P_{m_2-1},x\right)I\left(P_{m_3-1},x\right).
\end{eqnarray*}
Substituting $x=\frac{1}{4}$, and applying Proposition~\ref{prop:recur2},
\begin{eqnarray*}
I\left(B_{m_1,m_2,m_3},-\frac{1}{4}\right) & = & \left(\frac{1}{2^{m_1+2}}\right)\left(\frac{m_2+2}{2^{m_2}}\right)\left(\frac{m_3+2}{2^{m_3}}\right)\\
& & -\frac{1}{4}\left(\frac{1}{2^{m_1+1}}\right)\left(\frac{m_2+1}{2^{m_2-1}}\right)\left(\frac{m_3+1}{2^{m_3-1}}\right)\\
& = &\frac{(m_2+2)(m_3+2)-2(m_2+1)(m_3+1)}{2^{m_1+m_2+m_3+2}}\\
& = &\frac{2-m_2m_3}{2^{m_1+m_2+m_3+2}}.
\end{eqnarray*}
Since the denominator is positive, if $m_2m_3 \geq 2$ then $I\left(B_{m_1,m_2,m_3},-\frac{1}{4}\right) \leq 0$ so $I\left(B_{m_1,m_2,m_3},x\right)$ has a root in the interval $\left[-\frac{1}{4},0\right)$ and hence cannot be a factor of $I\left(P_n,x\right)$ for any $n$. Hence, we only need to consider the case where $m_2=m_3=1$.

Unfortunately, the graph $B_{m_1,1,1}$ is as complicated as $Y_{m_1,1,1}$ in the previous subsection. Let $p_m(x)=I(B_{m,1,1},x)$ for $m\geq 1$. The
We have
\begin{eqnarray*}
p_0(x) & = & 1 + 6x + 9x^2 + 3x^3, \\
p_1(x) & = & 1 + 7x + 14x^2 + 8x^3 + 2x^4.
\end{eqnarray*}
Using Proposition~\ref{prop:delvert} on one of the vertices of degree~1, we find that
\begin{eqnarray*}
p_m(x) & = & I(D_{m+5},x) + x(1+x)I(D_{m+3},x) \\
& = & I(D_{m+4},x) + xI(D_{m+3},x) \\
& & + x(1+x)\left(I(D_{m+2},x)+xI(D_{m+1},x)\right) \\
& = & I(D_{m+4},x)+x(1+x)I(D_{m+2},x)\\
& & + x\left(I(D_{m+3},x)+x(1+x)I(D_{m+1},x)\right) \\
& = & p_{m-1}(x) + xp_{m-2}(x)
\end{eqnarray*}
for $m\geq 2$. Therefore, the recurrence relation~(\ref{eq:recur}) applies to $p_m(x)$.

The solution to the recurrence relation for $p_m(x)$ is
$$p_m(x)=A(x)\left(\frac{1+\sqrt{1+4x}}{2}\right)^m+B(x)\left(\frac{1-\sqrt{1+4x}}{2}\right)^m$$
where
\begin{eqnarray*}
A(x) & = & \frac{1}{2}\left(1 + 6x + 9x^2 + 3x^3 + \frac{1 + 8x + 19x^2 + 13x^3 + 2x^4}{\sqrt{1+4x}}\right),\\
B(x) & = & \frac{1}{2}\left(1 + 6x + 9x^2 + 3x^3 - \frac{1 + 8x + 19x^2 + 13x^3 + 2x^4}{\sqrt{1+4x}}\right).
\end{eqnarray*}

Carrying out the same steps as for $Y_{m_1,1,1}$ and (again) substituting $x=-\frac{1}{4}(1+k^2)$, we arrive at
$$\tan\left(m \tan^{-1}(k)\right) = \frac{2k(1 - 33k^2 + 27k^4 - 3k^6)}{1 + 26k^2 - 80k^4 + 22k^6 - k^8}.$$

Now, the function
$$F(k)=\frac{1-33k^2+27k^4-3k^6}{k(11-43k^2+9k^4-k^6)}$$
for positive $k$ is not defined when $11-43k^2+9k^4-k^6=0$, which has one positive root $\alpha$ (which is approximately $0.52$). For $k\in(0,\alpha)$, $F(k)$ is strictly decreasing from $+\infty$ to $-\infty$. For $k>\alpha$, the behaviour of $F(k)$ is more complicated. It has a local minimum at $\beta\approx1.998$ and a local maximum at $\gamma\approx4.13$. For $k\in(\alpha,\beta]$, $F(k)$ is strictly decreasing from $+\infty$ to approximately $-0.673$.  For $k\in[\beta,\gamma]$, $F(k)$ is strictly increasing from approximately $-0.673$ to approximately $0.600$,and for $k\geq\gamma$, it is strictly decreasing from $0.600$ to approach $0$ asymptotically. The only region where $F(k)$ is increasing is $[\beta,\gamma]$, and in this region there is only one point of inflection, at approximately $k=2.592$. This is where $F^\prime(k)$ reaches its maximum value of approximately $1.281$.

The function
$$G_m(k)=\tan\left(m\tan^{-1}(k)\right)$$
has already been described in the previous subsection.

We would like to determine how many times the functions $F(k)$ and $G_m(k)$ take on the same value. Now for $m\geq 11$, it turns out that
$$G_m^\prime(k) = \frac{m}{1+k^2}\sec^2(m\tan^{-1}(k)) \geq \frac{m}{1+k^2} > F^\prime(k)$$
for all values of $k$ where both functions are defined.

I claim that for $m\geq 11$, the functions $F(k)$ and $G_m(k)$ take the same value $\frac{m+3}{2}$ times if $m$ is odd and $\frac{m+1}{2}$ times if $m$ is even. Note that $m$ is large enough that $\alpha\approx0.52$ falls in between some pair of values $\tan\left((2j-1)\frac{\pi}{2m}\right)$ and $\tan\left((2j+1)\frac{\pi}{2m}\right)$ for some integer $j$ where $1\leq j\leq \left\lfloor\frac{m}{2}\right\rfloor-1$. Now for $k<\tan\left(\frac{\pi}{2m}\right)$, $F(k)$ and $G_m(k)$ take the same value exactly once. This is because both functions are continuous between each pair of values, and $F(k)$ is decreasing while $G_m(k)$ is increasing. Between each pair of consecutive values listed in~(\ref{eq:kvalue}) from $\tan\left(\frac{\pi}{2m}\right)$ to $\tan\left((2j-1)\frac{\pi}{2m}\right)$, and from $\tan\left((2j+1)\frac{\pi}{2m}\right)$ to $\tan\left(\left(2\left\lfloor\frac{m}{2}\right\rfloor-1\right)\frac{\pi}{2m}\right)$, $F(k)$ and $G_m(k)$ take the same value exactly once. This is because both functions are continuous between each pair of values, $G_m^\prime(k)>F^\prime(k)$, and $G_m(k)$ is strictly increasing from $-\infty$ to $+\infty$. Between $\tan\left((2j-1)\frac{\pi}{2m}\right)$ and $\alpha$, $F(k)$ and $G_m(k)$ take the same value exactly once. This is because both functions are continuous between each pair of values, $F(k)$ is strictly decreasing from some finite value to $-\infty$, and $G_m(k)$ is strictly increasing from $-\infty$ to some finite value. Similarly, between $\alpha$ and $\tan\left((2j+1)\frac{\pi}{2m}\right)$, $F(k)$ and $G_m(k)$ take the same value exactly once. This is because both functions are continuous between each pair of values, $F(k)$ is strictly decreasing from $+\infty$ to some finite value, and $G_m(k)$ is strictly increasing from some finite value to $+\infty$. Finally, for $k>\tan\left(\left(2\left\lfloor\frac{m}{2}\right\rfloor-1\right)\frac{\pi}{2m}\right)$, if $m$ is odd, then  $F(k)$ and $G_m(k)$ take the same value exactly once. This is because $m$ is large enough that $\tan\left(\left(2\left\lfloor\frac{m}{2}\right\rfloor-1\right)\frac{\pi}{2m}\right)>\gamma$, so $G_m(k)$ is strictly increasing from $-\infty$ to $+\infty$ while $F(k)$ is strictly decreasing from some finite positive value to $0$. If $m$ is even, then $F(k)$ and $G_m(k)$ do not take the same value in that region as $F(k) > 0$ whereas $G_m(k) < 0$. Counting the total number of times $F(k)$ and $G_m(k)$ take the same value proves the claim.

Therefore, for $m\geq 11$, $p_m(x)=0$ has $\frac{m+3}{2}$ real roots less than $-\frac{1}{4}$ if $m$ is odd and $\frac{m+2}{2}$ real roots less than $-\frac{1}{4}$ if $m$ is even. However, we can see from the recurrence relation defining $p_m(x)$ that the degree of $p_m(x)$ is $\frac{m+3}{2}+2$ if $m$ is odd and $\frac{m+2}{2}+2$ if $m$ is even. Therefore, there have to be two more roots which are not real and less than $-\frac{1}{4}$. This eliminates $B_{m,1,1}$ as a possible connected component of $G$ for $m\geq 11$. For $1\leq m\leq 10$, we find that the independence polynomial has all real roots less than $-\frac{1}{4}$ for $m=0,5$ and no other values of $m$.

\subsection{The graphs $A_{m_1,m_2}$ and $E_{m_1,m_2}$}

It was already shown in Ng~\cite{ng} that $A_{m_1,m_2}$ is independence equivalent to both $E_{m_1,m_2}$ and $E_{m_2,m_1}$. This can be seen by applying Proposition~\ref{prop:delvert}. $E_{m_1,m_2}-u$ and $A_{m_1,m_2}-v$ are both isomorphic to $P_{m_1+m_2+2}$ and $E_{m_1,m_2}-N[u]$ and $A_{m_1,m_2}-N[v]$ are both isomorphic to $P_{m_1}\cup P_{m_2}$. Therefore, we will only consider $A_{m_1,m_2}$.

For $m_1\geq 3$, we can apply Proposition~\ref{prop:delvert} on $v_{m_1}$ to get
$$I(A_{m_1,m_2},x) = I(A_{m_1-1,m_2},x)+xI(A_{m_1-2,m_2},x),$$
so the recurrence relation~(\ref{eq:recur}) is satisfied, with $m_1$ as the index. We can then use the solution~(\ref{eq:recursol}) to determine the sign of $I\left(A_{m_1,m_2},-\frac{1}{4}\right)$.

When $m_2=1$ we find that $\tau_{m_1}=2^{-(m_1+5)}(4-m_1)$ so $I\left(A_{m_1,1},-\frac{1}{4}\right) \leq 0$ for $m_1\geq 4$.

When $m_2=2$ we find that $\tau_{m_1}=2^{-(m_1+5)}(2-m_1)$ so $I\left(A_{m_1,1},-\frac{1}{4}\right) \leq 0$ for $m_1\geq 2$.

Since $A_{m_1,m_2}$ and $A_{m_2,m_1}$ are isomorphic, we can use these two values to help us find $A_{m_1,m_2}$ in general. For general $m_2$, we have
$$\tau_1 = \frac{U+V}{2} = \frac{4-m_2}{2^{m_2+5}},\qquad\tau_2 = \frac{2U+V}{2} = \frac{2-m_2}{2^{m_2+5}}$$
so
$$I\left(A_{m_1,m_2},-\frac{1}{4}\right)=\tau_{m_1}=\frac{Um_1+V}{2^{m_1}}=\frac{4-m_1m_2}{2^{m_1+m_2+5}}.$$
Therefore, $I\left(A_{m_1,m_2},-\frac{1}{4}\right)>0$ only if $m_1m_2 < 4$. This leaves $A_{3,1}$, $A_{2,1}$ and $A_{1,1}$ as candidate connected components of $G$.

\subsection{Shortlisting the candidates}

Currently the list of graphs we can have as connected components $G_i$ of $G\in\mathcal{I}(P_n)$, are given in the table below. Their independence polynomials are also given in terms of $f_k(x)$ and $\tilde{f}_k(x)$ from Propositions~\ref{prop:cyclicfactor} and~\ref{prop:pathfactor}.

\begin{center}
\begin{tabular}{c|c|c}
Graph $G_i$ & \begin{tabular}{c}Independence\\polynomial $I(G_i,x)$\end{tabular} & Eliminate?\\
\hline
$P_k$ & $I(P_k,x)$ &  \\
$C_k$, $D_k$ & $I(C_k,x)$ &  \\
$Y_{z,2,1}$ where $z\geq 1$ & $\tilde{f}_3(x)I(C_{z+3},x)$ &  \\
$Y_{10,1,1}$ & $f_4(x)f_9(x)\tilde{f}_5(x)$ & Eliminate \\
$Y_{9,3,1}$ & $f_{21}(x)\tilde{f}_5(x)$ & Eliminate \\
$Y_{7,3,1}$ & $f_{15}(x)\tilde{f}_7(x)$ & Eliminate \\
$Y_{5,4,1}$ & $f_3(x)f_{15}(x)\tilde{f}_3(x)$ & Eliminate \\
$Y_{5,1,1}$ & $f_6(x)\tilde{f}_3(x)\tilde{f}_5(x)$ & Eliminate \\
$Y_{4,3,1}$ & $f_9(x)\tilde{f}_5(x)$ & Eliminate \\
$Y_{4,2,2}$ & $f_{12}(x)\tilde{f}_3(x)$ &  \\
$Y_{3,3,2}$ & $f_2(x)f_{15}(x)$ & Eliminate \\
$Y_{3,2,2}$ & $f_2(x)f_9(x)$ & Eliminate \\
$B_{5,1,1}$ & $f_4(x)f_{15}(x)$ & Eliminate \\
$B_{0,1,1}$, $E_{2,1}$, $E_{1,2}$, $A_{2,1}$ & $f_9(x)$ & \\
$E_{3,1}$, $E_{1,3}$, $A_{3,1}$ & $f_{15}(x)$ & \\
$E_{1,1}$, $A_{1,1}$ & $f_6(x)\tilde{f}_3(x)$ & \\
$K_4-e$ & $f_6(x)$ &
\end{tabular}
\end{center}

Returning to Proposition~\ref{prop:pathfactor}, we note that if $f_r(x)$ (respectively, $\tilde{f}_s(x)$) is a factor of $I(P_{n},x)$, then so are $f_{r^\prime}(x)$ (respectively, $\tilde{f}_{s^\prime}(x)$) for all $r^\prime\vert r$ (respectively, all $s^\prime\vert s$ where $s>1$). Furthermore, if $\tilde{f}_s(x)$ is a factor of $I(P_{n},x)$ then so is $f_s(x)$. The converse also holds if $s$ is odd. Similarly, from Proposition~\ref{prop:cyclicfactor}, if $f_r(x)$ is a factor of $I(C_n,x)$, then so are $f_{r^\prime}(x)$ for which $r^\prime$ is a factor of $r$, and $\frac{r}{r^\prime}$ is odd.

We can eliminate some of these as potential factors of $I(P_{n},x)$ for this reason. For instance, in the case of $Y_{10,1,1}$, the presence of $f_4(x)$ and $f_9(x)$ means that if $I(Y_{10,1,1},x)\vert I(P_{n},x)$, then $4\vert ((n+2)/2)$ and $9\vert ((n+2)/2)$, so $72\vert (n+2)$, and therefore, $f_{36}(x)\vert I(P_{n},x)$. However, it can be seen from the list of $I(G_i,x)$ in the table that it is not possible for $f_{36}(x)$ to be in a factor of $I(G_i,x)$ which is a factor of $I(P_{n},x)$, without $f_4(x)$ or $f_9(x)$ being also factors of the same $I(G_i,x)$. Hence, $I(Y_{10,1,1},x)$ cannot be one of the connected components of a graph $G$ independence equivalent to $P_{n}$.

We eliminate the following graphs based on similar arguments.
\begin{itemize}
\item $Y_{10,1,1}$, as mentioned above.
\item $Y_{9,3,1}$, since $f_{21}(x)\vert I(P_{n},x)$ and $\tilde{f}_5(x)\vert I(P_{n},x)$ together imply that $210\vert (n+2)$, but it is not possible for $f_{105}(x)\vert I(G_i,x)$ without $f_{21}(x)\vert I(G_i,x)$.
\item $Y_{7,3,1}$, since $f_{15}(x)\vert I(P_{n},x)$ and $\tilde{f}_7(x)\vert I(P_{n},x)$ together imply that $210\vert (n+2)$, but it is not possible for $f_{105}(x)\vert I(G_i,x)$ without $f_{15}(x)\vert I(G_i,x)$.
\item $Y_{5,4,1}$, since $f_{15}(x)\vert I(P_{n},x)$ implies that $30\vert (n+2)$, but it is not possible for $\tilde{f}_{15}(x)\vert I(G_i,x)$ without $\tilde{f}_3(x)\vert I(G_i,x)$.
\item $Y_{5,1,1}$, since $\tilde{f}_3(x)\vert I(P_{n},x)$ and $\tilde{f}_5(x)\vert I(P_{n},x)$ together imply that $15\vert (n+2)$, but it is not possible for $\tilde{f}_{15}(x)\vert I(G_i,x)$ without $\tilde{f}_3(x)\vert I(G_i,x)$.
\item $Y_{4,3,1}$, since $f_{9}(x)\vert I(P_{n},x)$ and $\tilde{f}_5(x)\vert I(P_{n},x)$ together imply that $90\vert (n+2)$, but it is not possible for $\tilde{f}_{45}(x)\vert I(G_i,x)$ without $\tilde{f}_5(x)\vert I(G_i,x)$.
\item $Y_{3,3,2}$, since $f_2(x)\vert I(P_{n},x)$ and $f_{15}(x)\vert I(P_{n},x)$ together imply that $60\vert (n+2)$, but it is not possible for $f_{30}(x)\vert I(G_i,x)$ without $f_2(x)\vert I(G_i,x)$.
\item $Y_{3,2,2}$, since $f_2(x)\vert I(P_{n},x)$ and $f_9(x)\vert I(P_{n},x)$ together imply that $36\vert (n+2)$, but it is not possible for $f_{18}(x)\vert I(G_i,x)$ without $f_2(x)\vert I(G_i,x)$.
\item $B_{5,1,1}$, since $f_4(x)\vert I(P_{n},x)$ and $f_{15}(x)\vert I(P_{n},x)$ together imply that $120\vert (n+2)$, but it is not possible for $f_{60}(x)\vert I(G_i,x)$ without $f_4(x)\vert I(G_i,x)$.
\end{itemize}
Furthermore, in the case of $Y_{z,2,1}$, the presence of $\tilde{f}_3(x)$ means that $3\vert (n+2)$, hence $f_3(x)\vert (n+2)$, and hence, $6\vert (n+2)$.

We have indicated the graphs to be eliminated in the table.

\section{Final result}

We are now ready to list the graphs which can be in $\mathcal{I}(P_{n})$, where $n$ is even.

Suppose that $n+2=2^t m$, where $m$ is odd. Now, $\tilde{f}_m(x)\vert I(G,x)$, and for any $s\vert m$, we also have $\tilde{f}_s(x)\vert I(G,x)$. Examining the table of candidate graphs for $G_i$, it is not possible for $\tilde{f}_m(x)$ to be a factor of $I(G_i,x)$ without $\tilde{f}_s(x)$ being also a factor of the same $I(G_i,x)$. Hence, there exists a $G_i$ such that
$$\left. \prod_{s\vert m}\tilde{f}_s(x)\right\vert I(G_i,x).$$

\subsection{Case 1: $n=2^t m -2$, where $m\neq 3$}

Examining the table of candidate connected graphs $G_i$ and the factorisation of paths and cycle graphs in Propositions~\ref{prop:cyclicfactor} and~\ref{prop:pathfactor}, we find that for $m$ in Case~1,
the only possible graphs for $G_i$ are paths $P_{2^{t^\prime}m-2}$ for some $0\leq t^\prime < t$.

The remaining factors of $I(G,x)$ that are not factors of $I(P_{2^{t^\prime}m-2},x)$ are of the form $f_{2^j r}(x)$, where $t^\prime \leq j \leq t-1$ and $r\vert m$.

\textbf{Subcase~1.1: $n=2^t m -2$, where $m\neq 9$ and $m\neq 15$}

For each $j\in\{t^\prime,\ldots,t-1\}$, $f_{2^j m}(x)\vert I(G_i,x)$ for some $G_i$ that is not $P_{n^\prime}$. Examining the table of candidate graphs for $G_i$, the only possibilities are the cycle graphs or $D_k$. Since $D_k$ is independence equivalent to the cycle graph $C_k$, we will only consider the cycle graphs.

Now if $f_{2^j m}(x)\vert I(C_k,x)$, then $f_{2^j r}(x)\vert I(C_k,x)$ for all $r\vert m$. Hence
$$I(C_{2^j m},x)=\left.\prod_{r\vert m} f_{2^j m}(x)\right\vert I(C_k,x).$$
By Corollary~\ref{cor:cyclicfactor}, $k/m$ must be odd. On the other hand, there are no factors of $I(P_{n},x)$ of the form $f_{k}(x)$ where $k/m$ is an odd number greater than~$1$, hence $k=m$ and so $G_i=C_{2^j m}$.

Since this is true for each $j=t^\prime,\ldots,t-1$, we have each cycle graph $C_{2^j m}$ as one of the connected components of $G$. Therefore, the graphs which are independence equivalent to $P_{n}$, where $n=2^t m-2$, $m$ odd, are of the form
$$Z_{n,t-t^\prime} = C_{2^{t-1}m} \cup C_{2^{t-2}m} \cup \cdots \cup C_{2^{t^\prime}m} \cup  P_{2^{t^\prime} m-2} $$
for some $t^\prime < t$ (recall Definition~\ref{def:mathcalp}).
Therefore,
$$\mathcal{I}(P_{n})=\{P_n\}\cup\bigcup_{t^\prime = 0}^{t-1} \mathcal{D}(Z_{n,t-t^\prime})=\mathcal{P}_n.$$

\textbf{Subcase 1.2: $n=9\times 2^t - 2$}

Since $\tilde{f}_9(x)\vert I(Gx)$, one of the connected components of $G$ must be $P_{9\times2^{t^\prime}-2}$ for some $t^\prime < t$.

Now if $t^\prime \geq 1$, then the remaining factors of $I(G,x)$ that are not factors of $I(P_{9\times2^{t^\prime}-2},x)$ are of the form $f_{2^j r}(x)$, where $1 < t^\prime \leq j \leq t-1$ and $r\vert 9$, and so the only possibilities for the remaining $G_i$ are cycle graphs or $D_k$. Similarly to Subcase~1.1, we find that the graphs which are independence equivalent to $P_{n}$, where $n=9\times 2^t -2$, are of the form
$$Z_{n,t-t^\prime} = C_{9\times2^{t-1}} \cup C_{9\times2^{t-2}m} \cup \cdots \cup C_{9\times2^{t^\prime}} \cup  P_{9\times2^{t^\prime}-2}$$
for some $t^\prime < t$ (recall Definition~\ref{def:mathcalp}).

If $t^\prime =0$, then the remaining factors of $I(G,x)$ that are not factors of $I(P_{9\times2^{t^\prime}-2},x)$ are of the form $f_{2^j r}(x)$, where $0 \leq j \leq t$ and $r\vert 9$. For each $j \geq 1$, following the same argument above, we must have the cycle graph $C_{9 \times 2^j}$. However, for $j=0$, we can replace the $C_9$ component with another member of $\mathcal{I}(C_9)$. In other words, $G\in\mathcal{D}(Z_{n,t})$ (recall Definition~\ref{def:mathcald}).

Therefore,
$$\mathcal{I}(P_{n})=\{P_n\}\cup\bigcup_{t^\prime = 0}^{t-1} \mathcal{D}(Z_{n,t-t^\prime})=\mathcal{P}_n.$$

\textbf{Subcase 1.3: $n=15\times 2^t - 2$}

Since $\tilde{f}_{15}(x)\vert I(Gx)$, one of the connected components of $G$ must be $P_{15\times 2^{t^\prime}-2}$ for some $t^\prime < t$.

Now if $t^\prime \geq 1$, then the remaining factors of $I(G,x)$ that are not factors of $I(P_{15\times 2^{t^\prime}-2},x)$ are of the form $f_{2^j r}(x)$, where $1 < t^\prime \leq j \leq t-1$ and $r\vert 15$, and so the only possibilities for the remaining $G_i$ are cycle graphs or $D_k$. Similarly to Subcase~1.1, we find that the graphs which are independence equivalent to $P_{n}$, where $n=15 \times 2^t -2$, are of the form
$$Z_{n,t-t^\prime} = C_{15\times2^{t-1}} \cup C_{15\times2^{t-2}} \cup \cdots \cup C_{15\times2^{t^\prime}} \cup  P_{15\times2^{t^\prime}-2}$$
for some $t^\prime < t$ (recall Definition~\ref{def:mathcalp}).

If $t^\prime =0$, then the remaining factors of $I(G,x)$ that are not factors of $I(P_{15\times 2^{t^\prime}-2},x)$ are of the form $f_{2^j \times r}(x)$, where $0 \leq j \leq t$ and $r\vert 15$. For each $j \geq 1$, following the same argument above, we must have the cycle graph $C_{15 \times 2^j}$. However, for $j=0$, we can replace the $C_{15}$ component with another member of $\mathcal{I}(C_{15})$. In other words, $G\in\mathcal{D}(Z_{n,t})$ (recall Definition~\ref{def:mathcald}).

Therefore,
$$\mathcal{I}(P_{n})=\{P_n\}\cup\bigcup_{t^\prime = 0}^{t-1} \mathcal{D}(Z_{n,t-t^\prime})=\mathcal{P}_n.$$

\subsection{Case~2: $n=3\times 2^t - 2$}

In this case, $\tilde{f}_3(x)|I(G,x)$. We consider various subcases for the graph $G_i$ such that $\tilde{f}_3(x)|I(G_i,x)$.

\textbf{Subcase 2.1: $G_i=P_{3\times2^{t^\prime}-2}$ for some $0\leq t^\prime < t$}

The remaining factors of $I(G,x)$ that are not factors of $I(P_{3\times2^{t^\prime}-2},x)$ are of the form $f_{2^j}(x)$ or $f_{3\times 2^j}(x)$, where $t^\prime \leq j \leq t-1$.

If $t^\prime \geq 2$ then the only possibilities for the remaining $G_i$ are cycle graphs or $D_k$. Similarly to Subcase~1.1, we find that the graphs which are independence equivalent to $P_{n}$, where $n=3\times 2^t$, are of the form
$$Z_{n,t-t^\prime} = C_{3\times2^{t-1}} \cup C_{3\times2^{t-2}} \cup \cdots \cup C_{3\times2^{t^\prime}} \cup  P_{3\times2^{t^\prime}-2}$$
for some $t^\prime < t$ (recall Definition~\ref{def:mathcalp}).

If $t^\prime \leq 1$ then the remaining factors of $I(G,x)$ that are not factors of $I(P_{3\times2^{t^\prime}-2},x)$ are of the form $f_{2^j}(x)$ or $f_{3\times 2^j}(x)$, where $t^\prime \leq j \leq t$. For each $j \geq 2$, following the same argument above, we must have the cycle graph $C_{3\times2^j}$. However, for $j\leq1$, we can replace the $C_6$ component with another member of $\mathcal{I}(C_6)$. In other words, $G\in\mathcal{D}(Z_{n,t})$ (recall Definition~\ref{def:mathcald}).

\textbf{Subcase 2.2: $G_i=Y_{4,2,2}$}

Since $f_{12}(x)\vert I(Y_{4,2,2},x)\vert I(P_{n},x)$, we have $24\vert (n+2)$ so $t\geq 3$.

The remaining factors of $I(P_{n},x)=I(G,x)$ that are not factors of $I(Y_{4,2,2},x)$ are of the form $f_{2^j}(x)$ or $f_{2^j \times 3}(x)$ for $3\leq j < t$, as well as $f_2(x)$, $f_3(x)$, $f_4(x)$ and $f_6(x)$.

Now for each $j\geq 3$, following the same argument above, we must have the cycle graph $C_{3\times2^j}$. For the remaining factors, the connected graphs to consider are $P_2$, $P_6$, $C_3$, $C_4$, $C_6$ and $K_4-e$. Hence, we have the following possibilities:
\begin{itemize}
\item $Y_{4,2,2}\cup P_2 \cup C_3 \cup C_4 \cup (K_4-e) \cup \displaystyle\bigcup_{j=3}^{t-1} C_{2^j\times 3}$
\item $Y_{4,2,2}\cup C_3 \cup C_4 \cup C_6 \cup \displaystyle\bigcup_{j=3}^{t-1} C_{2^j\times 3}$
\item $Y_{4,2,2} \cup C_3 \cup P_6 \cup (K_4-e) \cup \displaystyle\bigcup_{j=3}^{t-1} C_{2^j\times 3}$
\end{itemize}
where $\bigcup_{j=3}^{t-1} C_{3\times2^j}$ is the empty graph if $t=2$. These are the graphs that are used to define $\mathcal{Y}^{(4)}_n$ and $\mathcal{Y}^{(5)}_n$. Therefore, $G\in \mathcal{Y}^{(4)}_n\cup\mathcal{Y}^{(5)}_n$ (recall Definition~\ref{def:mathcaly45}).

\textbf{Subcase 2.3: $G_i=E_{1,1}$ or $A_{1,1}$}

Since $f_{6}(x)\vert I(G_i,x)\vert I(P_{n},x)$, we have $12\vert (n+2)$ so $t\geq 2$.

The remaining factors of $I(P_{n},x)=I(G,x)$ that are not factors of $I(G_i,x)$ are of the form $f_{2^j}(x)$ or $f_{2^j \times 3}(x)$ for $2\leq j < t$, as well as $f_2(x)$ and $f_3(x)$.

Now for each $j\geq 2$, following the same argument above, we must have the cycle graph $C_{3\times2^j}$. For the remaining factors, the $G_i$ to consider are $P_2$ and $C_3$. Hence, we have the following possibilities:
\begin{itemize}
\item $E_{1,1}\cup P_2 \cup C_3 \cup \displaystyle\bigcup_{j=2}^{t-1} C_{2^j\times 3}$
\item $A_{1,1}\cup P_2 \cup C_3 \cup \displaystyle\bigcup_{j=2}^{t-1} C_{2^j\times 3}$
\end{itemize}
where $\bigcup_{j=2}^{t-1} C_{3\times2^j}$ is the empty graph if $t=2$. These are the graphs that are used to define $\mathcal{Y}^{(2)}_n$ and $\mathcal{Y}^{(3)}_n$. Therefore, $G\in \mathcal{Y}^{(2)}_n\cup\mathcal{Y}^{(3)}_n$ (recall Definition~\ref{def:mathcaly23}).

\textbf{Subcase 2.4: $G_i=Y_{z,2,1}$}

Since $I(Y_{z,2,1},x)=I(P_1,x)I(C_{z+3},x)$, we need only look for instances where $P_1$ and $C_{z+3}$ appear as connected components in the same graph~$G$. It turns out these graphs all appeared in Subcase~2.1 and in none of the other subcases. This leads to the set $\mathcal{Y}^{(1)}_n$ (recall Definition~\ref{def:mathcaly1}). Therefore, $G\in \mathcal{Y}^{(1)}_n$ (recall Definition~\ref{def:mathcaly1}).

\subsection{Conclusion}

This concludes the proof of Theorem~\ref{prop:pathequiv}, and provides a complete solution to Problem~1 of~\cite{beaton}, which asks for the independence equivalence class of paths of even orders. It may also be pointed out that the methods used to construct the table (in particular, the inclusion of $B_{0,1,1}$ and the various $A_{m_{1},m_{2}}$ and $E_{m_{1},m_{2}}$ graphs) provide an alternative proof of Theorem~6 of~\cite{ng}.

%\bibliography{mybibfile}
\printbibliography

@article{hoedeli,
  author =       "C. Hoede and X. Li",
  title =        "Clique polynomials and independent set polynomials of graphs",
  journal =      "Discrete Math.",
  volume =       "125",
  number =       "1--3",
  pages =        "219--228",
  year =         "1994",
  DOI =          "https://doi.org/10.1016/0012-365X(94)90163-5"
}

@article{zhang,
  author =       "H. Zhang",
  title =        "A way to construct independence equivalent graphs",
  journal =      "Appl. Math. Letters",
  volume =       "25",
  number =       "10",
  pages =        "1304--1308",
  year =         "2012",
  DOI =          "https://doi.org/10.1016/j.aml.2011.11.033"
}

@article{beaton,
  author =       "I. Beaton and J.I. Brown and B. Cameron",
  title =        "Independence equivalence classes of paths and cycles",
  journal =      "Australas. J. Combin.",
  volume =       "75",
  number =       "1",
  pages =        "127--145",
  year =         "2019"
}

@article{alikhani,
  author =       "S. Alikhani and Y.-H. Peng",
  title =        "Independence roots and independence fractals of certain graphs",
  journal =      "J. Appl. Math. Comput.",
  volume =       "36",
  number =       "1--2",
  pages =        "89--100",
  year =         "2011",
  DOI =          "https://doi.org/10.1007/s12190-010-0389-4"
}

@article{lehmer,
  author =       "D.H. Lehmer",
  title =        "A note on trigonometric algebraic numbers",
  journal =      "Amer. Math. Monthly",
  volume =       "40",
  number =       "3",
  pages =        "165--166",
  year =         "1933",
  DOI =          "https://doi.org/10.2307/2301023"
}

@article{watkins,
  author =       "W. Watkins and J. Zeitlin",
  title =        "The minimal polynomial of $\cos(2\pi/n)$",
  journal =      "Amer. Math. Monthly",
  volume =       "100",
  number =       "5",
  pages =        "471--474",
  year =         "1993",
  DOI =          "https://doi.org/10.2307/2324301"
}

@phdthesis{chism,
  author = "L.M. Chism",
  title  = "On independence polynomials and independence equivalence in graphs",
  school = "University of Mississippi",
  year = "2009"
}

@phdthesis{wingard,
  author = "G.C. Wingard",
  title  = "Properties and applications of the Fibonacci polynomial of a graph",
  school = "University of Mississippi",
  year = "1995"
}

@inproceedings{levit,
  author =       "V.E. Levit and E. Mandrescu",
  title =        "The independence polynomial of a graph -- a survey",
  booktitle =    "Proceedings of the 1st International Conference on Algebraic Informatics. Held in Thessaloniki, October 20-23, 2005",
  pages =        "233--254",
  year =         "2005",
  address =      "Thessaloniki",
  publisher =    "Aristotle Univ."
}

@article{ng,
  title= "Independence equivalence classes of cycles", 
  author= "B.L. Ng",
  journal =      "Discrete Math.",
  volume =       "344",
  number =       "12",
  pages =        "112605",
  year =         "2021",
  DOI =          "https://doi.org/10.1016/j.disc.2021.112605"
}

@book{gross,
 author    = "Gross, {J.L.} and J. {Yellen} and M. {Anderson}",
 title     = "Graph Theory and its Applications",
 publisher = "CRC Press",
 year      =  2018,
 edition   = "3rd"
}

@article{huawang,
  author =       "S. Wagner and H. Wang",
  title =        "Indistinguishable trees and graphs",
  journal =      "Graph. Combin.",
  volume =       "30",
  number =       "6",
  pages =        "1593--1605",
  year =         "2013",
  DOI =          "https://doi.org/10.1007/s00373-013-1360-6"
}

@article{makowsky,
  author =       "J.A. Makowsky and V. Rakita",
  title =        "Weakly distinguishing graph polynomials on addable properties",
  journal =      "Mosc. J. Comb. Number Theory",
  volume =       "9",
  number =       "3",
  pages =        "333--349",
  year =         "2020",
  DOI =          "https://doi.org/10.2140/moscow.2020.9.333"
}
\end{document}